\documentclass[11pt,reqno,a4paper]{amsart}
\usepackage{amsmath,amsthm,verbatim,amscd,amssymb,setspace,enumitem,hyperref,esint,mathtools}
\usepackage{exscale,color}

\usepackage{enumitem}
\usepackage{mathrsfs}

\renewcommand{\epsilon}{\varepsilon}

\newcommand{\Z}{\mathbb{Z}}

\newcommand{\R}{\mathbb{R}}
\newcommand{\C}{\mathbb{C}}
\renewcommand{\P}{\mathbb{P}}

\newcounter{mtheorem}
\newtheorem{mtheorem}[mtheorem]{Theorem}

\renewcommand{\P}{\mathbb{P}}
\newcommand{{\vol}}{\rm vol}
\newcommand{\p}{\partial}
\newcommand{\norm}[1]{\Vert #1 \Vert}

\newcommand{\Ric}{\operatorname{Ric}}

\def\RR{\operatorname{R}}

\providecommand{\norm}[1]{\lVert#1\rVert}

\def \Cstarn{(\mathbb{C}^*)^n}
\def \Cstarsqr{(\mathbb{C}^*)^2}
\def \Cstar{\mathbb{C}^*}
\def \t {\mathfrak{t}}
\def \bp {\bar{\partial}}

\def\Ric{\operatorname{Ric}}

\def\RR{\operatorname{R}}

\def\vol{\operatorname{vol}}

\newtheoremstyle{fancy}{}{}{\itshape}{}{\textbf\bgroup}{.\egroup}{ }{}
\newtheoremstyle{fancy2}{}{}{\rm}{}{\textbf\bgroup}{.\egroup}{ }{}

\theoremstyle{fancy}
\newtheorem{theorem}{Theorem}[section]
\newtheorem{lemma}[theorem]{Lemma}
\newtheorem{corollary}[theorem]{Corollary}

\newtheorem{prop}[theorem]{Proposition}
\newtheorem{conj}[theorem]{Conjecture}

\theoremstyle{fancy2}
\newtheorem{definition}[theorem]{Definition}
\newtheorem{example}[theorem]{Example}

\newtheorem{claim}[theorem]{Claim}

\setlist{leftmargin=*}

\numberwithin{equation}{section}

\begin{document}
\title{On finite time Type I singularities of the K\"ahler-Ricci flow on compact K\"ahler surfaces}
\date{\today}

\author{Charles Cifarelli}
\address{Laboratoire de Math\'ematiques Jean Leray (UMR 6629), Universit\'e de Nantes, CNRS,
2, rue de la Houssini\`ere, B.P.~92208, 44322 Nantes Cedex~3, France}
\email{Charles.Cifarelli@univ-nantes.fr}
\author{Ronan J.~Conlon}
\address{Department of Mathematical Sciences, The University of Texas at Dallas, Richardson, TX 75080}
\email{ronan.conlon@utdallas.edu}
\author{Alix Deruelle}
\address{Sorbonne Universit\'e and Universit\'e de Paris, CNRS, IMJ-PRG, F-75005 Paris, France}
\email{alix.deruelle@imj-prg.fr}

\date{\today}

\begin{abstract}
We show that the underlying complex manifold of a complete non-compact two-\linebreak dimensional shrinking gradient K\"ahler-Ricci soliton $(M,\,g,\,X)$ with soliton metric $g$ with
bounded scalar curvature $\RR_{g}$ whose soliton vector field $X$ has an integral curve along which $\RR_{g}\not\to0$
is biholomorphic to either $\mathbb{C}\times\mathbb{P}^{1}$ or to the blowup of this manifold at one point. Assuming the existence of such a soliton on
this latter manifold, we show that it is toric and unique. We also identify the corresponding soliton vector field. Given these possibilities, we then prove a strong form of the Feldman-Ilmanen-Knopf conjecture for finite time Type I singularities of the K\"ahler-Ricci flow on compact K\"ahler surfaces, leading to a classification of the bubbles
of such singularities in this dimension.
\end{abstract}

\maketitle

\markboth{Charles Cifarelli, Ronan J.~Conlon, and Alix Deruelle}{On finite time Type I singularities of the K\"ahler-Ricci flow on compact K\"ahler surfaces}

\section{Introduction}

\subsection{Overview}

 A \emph{Ricci soliton} is a triple $(M,\,g,\,X)$, where $M$ is a Riemannian manifold endowed with a complete Riemannian metric $g$
and a complete vector field $X$, such that
\begin{equation}\label{soliton111}
\Ric(g)+\frac{1}{2}\mathcal{L}_{X}g=\frac{\lambda}{2}g
\end{equation}
for some $\lambda\in\{-1,\,0,\,1\}$. If $X=\nabla^{g} f$ for some smooth real-valued function $f$ on $M$,
then we say that $(M,\,g,\,X)$ is \emph{gradient}. In this case, the soliton equation \eqref{soliton111}
becomes $$\Ric(g)+\frac{\lambda}{2} g=\operatorname{Hess}(f).$$

If $g$ is complete and K\"ahler with K\"ahler form $\omega$, then we say that $(M,\,g,\,X)$ is a \emph{K\"ahler-Ricci soliton} if
the vector field $X$ is complete and real holomorphic and the pair $(g,\,X)$ satisfies the equation
 \begin{equation}\label{soliton13}
\Ric(g)+\frac{1}{2}\mathcal{L}_{X}g=\lambda g
\end{equation}
for $\lambda$ as above. If $g$ is a K\"ahler-Ricci soliton and if $X=\nabla^{g} f$ for some smooth real-valued function $f$ on $M$,
then we say that $(M,\,g,\,X)$ is \emph{gradient}. In this case, the soliton equation \eqref{soliton13} may be rewritten as
\begin{equation*}
\rho_{\omega}+i\partial\bar{\partial}f=\lambda\omega,
\end{equation*}
where $\rho_{\omega}$ is the Ricci form of $\omega$.

For Ricci solitons and K\"ahler-Ricci solitons $(M,\,g,\,X)$, the vector field $X$ is called the
\emph{soliton vector field}. Its completeness is guaranteed by the completeness of $g$
\cite{Zhang-Com-Ricci-Sol}. If the soliton is gradient, then
the smooth real-valued function $f$ satisfying $X=\nabla^g f$ is called the \emph{soliton potential}. It is unique up to addition of a constant.
Finally, a Ricci soliton and a K\"ahler-Ricci soliton are called \emph{steady} if $\lambda=0$, \emph{expanding}
if $\lambda=-1$, and \emph{shrinking} if $\lambda=1$ in \eqref{soliton111} and \eqref{soliton13}, respectively.

The study of Ricci solitons and their classification is important in the context of Riemannian geometry. For example, they provide a
natural generalisation of Einstein manifolds and on certain Fano manifolds, shrinking K\"ahler-Ricci solitons are known to exist where
there are obstructions to the existence of a K\"ahler-Einstein metric \cite{zhuu}. Also, to each soliton, one may associate a self-similar solution of the Ricci flow
\cite[Lemma 2.4]{Chowchow}. These are candidates for singularity models of the flow. The difference in
normalisations between \eqref{soliton111} and \eqref{soliton13} reflects the difference between the constants preceding the Ricci term
in the Ricci flow and in the K\"ahler-Ricci flow respectively when one takes this dynamic point of view.

In this article we are concerned with the classification of complete shrinking gradient K\"ahler-Ricci solitons with
bounded curvature, the motivation being that such a soliton encodes how the K\"ahler-Ricci flow enters a finite time Type I singularity, that is, a singularity
where the curvature of the evolving metric doesn't blow up faster than $O((T-t)^{-1})$ at the finite singular time $T>0$. More precisely, non-flat shrinking gradient K\"ahler-Ricci solitons are known to appear as parabolic rescalings of finite time Type I singularities of the K\"ahler-Ricci flow on compact K\"ahler manifolds \cite{topping, naber}. We focus on the classification in complex dimension $2$, where a bound on the scalar curvature of the soliton suffices to bound the full curvature tensor \cite{wang22}. Assuming therefore bounded scalar curvature,
the soliton is either compact, in which case the underlying manifold is Fano and the resulting soliton is (up to automorphism) K\"ahler-Einstein or the shrinking gradient K\"ahler-Ricci soliton given by \cite{soliton} depending on the Fano manifold in question, or is non-compact. Gradient shrinking K\"ahler-Ricci solitons are connected at infinity \cite{munteanu}
and in this latter case, there is a dichotomy in the sense that the scalar curvature of the soliton either tends to zero along every integral curve of $X$, or
$X$ has an integral curve along which the scalar curvature does not tend to zero.
In the former case, it follows that the scalar curvature tends to zero globally (cf.~Lemma \ref{batman}) and hence
the soliton (up to automorphism) is either that of Feldman-Ilmanen-Knopf \cite{FIK} on the blowup of $\mathbb{C}^{2}$ at one point or the flat Gaussian shrinking soliton on $\mathbb{C}^{2}$ \cite{cds}. Here we use a result of \cite{naber} to prove, in conjunction with \cite{charlie}, that in the latter case the shrinking soliton is isometric to the cylinder $\mathbb{C}\times\mathbb{P}^{1}$ or to a new hypothetical toric shrinking gradient K\"ahler-Ricci soliton on the blowup of this latter manifold at one point. Being the only possibilities, this allows us to prove a strong form of the Feldman-Ilmanen-Knopf conjecture \cite{FIK} for finite time Type I singularities of the K\"ahler-Ricci flow on compact K\"ahler surfaces, and in doing so, identify the possible parabolic rescalings that may appear at such singularities.

\subsection{Main results}

The simplest examples of complete shrinking gradient K\"ahler-Ricci solitons include any K\"ahler-Einstein manifold with soliton vector field
$X=0$ and the flat Gaussian shrinking soliton on $\mathbb{C}$ endowed with soliton vector field $2\cdot\operatorname{Re}(z\partial_{z})$, $z$ here the holomorphic coordinate on $\mathbb{C}$.
Taking Cartesian products also provides examples. With this in mind, our first main result can be stated as follows. The statement should be read in the context of the dichotomy explained above.
\begin{mtheorem}[Holomorphic classification]\label{mainthm1}
Let $(M,\,g,\,X)$ be a two-dimensional complete non-compact shrinking gradient K\"ahler-Ricci soliton with complex structure $J$ and with
bounded scalar curvature $\RR_{g}$ whose soliton vector field $X$ has an integral curve along which $\RR_{g}\not\to0$. Then:
\begin{enumerate}
\item $M$ is biholomorphic to either $\mathbb{C}\times\mathbb{P}^{1}$ or to $\operatorname{Bl}_{p}(\C \times \P^1)$, that is, the blowup of $\mathbb{C}\times\mathbb{P}^{1}$ at a fixed point $p$ of the standard torus action on $\mathbb{C}\times\mathbb{P}^{1}$.
\item There exists a biholomorphism $\gamma:M\to M$ such that $\gamma^{-1}_{*}(JX)$ lies in the Lie algebra of the
real torus $\mathbb{T}$ acting on these spaces in the standard way and $\gamma^{*}g$ is $\mathbb{T}$-invariant.
\item $\gamma^{-1}_{*}(JX)$ is determined and its flow generates a holomorphic isometric $S^{1}$-action of $(M,\,J,\,\gamma^{*}g)$.
\item Assuming existence, $\gamma^{*}g$ is the unique $\mathbb{T}$-invariant complete shrinking gradient K\"ahler-Ricci soliton on $M$.
\end{enumerate}
\end{mtheorem}

Conclusions (ii)--(iv) for $M=\mathbb{C}\times\mathbb{P}^{1}$ have already been established in \cite{charlie} where it is shown that
any complete shrinking gradient K\"ahler-Ricci soliton with bounded scalar curvature on this manifold is isometric to
the Cartesian product of the flat Gaussian soliton $\omega_{\mathbb{C}}$ on $\mathbb{C}$ and twice the Fubini-Study metric $\omega_{\mathbb{P}^{1}}$ on $\mathbb{P}^{1}$. The new possibility arising is when $M$ is the blowup of $\mathbb{C}\times\mathbb{P}^{1}$ at one point, in which case $\gamma^{-1}_{*}(JX)$ is given by \eqref{vfield}. In light of this, we make the following conjecture.

\begin{conj}\label{conjecture}
There exists a complete shrinking gradient K\"ahler-Ricci soliton $\omega$ on $\operatorname{Bl}_{p}(\C \times \P^1)$, that is, the blowup of $\mathbb{C}\times\mathbb{P}^{1}$ at a fixed point $p$ of the standard torus action on $\mathbb{C}\times\mathbb{P}^{1}$, invariant under the real torus action induced by the standard real torus action on $\mathbb{C}\times\mathbb{P}^{1}$,
with bounded scalar curvature and with soliton vector field given by \eqref{vfield}. Moreover, there exists a biholomorphism $\Phi$ of $\operatorname{Bl}_{p}(\C \times \P^1)$ such that $\Phi^*\omega$ converges to $\frac{i}{2}\partial\bar{\partial}|z|^{\frac{1}{\lambda}}+2\omega_{\mathbb{P}^{1}}$ at a polynomial rate with respect to the radial coordinate $z$ on the $\mathbb{C}$ factor. Here, $0<\lambda<1$ is as in \eqref{vfield}.
\end{conj}

This conjecture will be explored in the forthcoming \cite{ccd2}. Its resolution, combined with Theorem \ref{mainthm1}, \cite[Corollary C]{charlie}, and \cite[Theorem E(3)]{cds}, would complete the classification of complete shrinking gradient K\"ahler-Ricci solitons with bounded scalar curvature in complex dimension $2$. The scaling factor in the $\mathbb{C}$-direction of the model at infinity is a result of the model metric having to be compatible with the pre-determined soliton vector field. We expect this soliton to model a finite time Type I collapsing singularity of the K\"ahler-Ricci flow in the vicinity of a $(-1)$-curve on a Fano surface with diameter bounded uniformly from below along the flow.

The proof of Theorem \ref{mainthm1} is specifically catered to complex dimension $2$, making heavy use of the theory of $J$-holomorphic curves in this dimension.
The outline is as follows. We assume that the shrinking soliton $(M,\,g,\,X)$ is simply connected as this turns out to suffice. The bounded scalar curvature assumption implies
that the curvature is bounded \cite{wang22} and so by results in \cite{cds}, the flow of $JX$ will generate the holomorphic isometric action of a real torus on the soliton.
Next, we are able to deduce from a result of Naber \cite{naber} that on large balls sufficiently far away from the zero set of the soliton vector field $X$
centred along the integral curve of $X$ along which $\RR_{g}\not\to0$, $(M,\,g)$ is $C^{\infty}$-close to the model cylinder $\mathbb{C}\times\mathbb{P}^{1}$. As the complex structures will consequently also be close, we use the perturbation theory of $J$-holomorphic curves to perturb a holomorphic $\mathbb{P}^{1}$ in the cylinder to a holomorphic $\mathbb{P}^{1}$ with zero self-intersection in $M$ itself. Taking an $S^{1}$ inside the aforementioned real torus generated by $JX$, we can then move this $\mathbb{P}^{1}$ around and identify $M$ with $\mathbb{C}^{*}\times\mathbb{P}^{1}$ at infinity. Complete shrinking solitons with bounded scalar curvature have finite topological type \cite{fang}, therefore we may blow down all of the $(-1)$-curves in $M$ and obtain its minimal model $M_{\min}$. A continuity argument using Gromov's compactness theorem for $J$-holomorphic curves then allows us to extend the $\mathbb{P}^{1}$-foliation of $M$ at infinity into the interior of $M_{\min}$ and in doing so, identify $M_{\min}$ as a $\mathbb{P}^{1}$-bundle over a non-compact Riemann surface $S$. Compactifying this picture, the assumption of simple connectedness allows us to ascertain that $S$ compactifies to an $S^{2}$, leaving us with the diffeomorphism type of $M_{\min}$ as $\mathbb{R}^{2}\times S^{2}$. After analysing the structure of the zero set of $X$, we may then use the flow of the vector fields $X$ and $JX$
to construct a complex torus equivariant biholomorphism between $M_{\min}$ and $\mathbb{C}\times\mathbb{P}^{1}$. $M$ is therefore biholomorphic to either $M_{\min}$ or to the blowup of $M_{\min}$ at finitely many points. The blow-up points of $M_{\min}$ must be contained in the zero set of the vector field that $X$ induces on $M_{\min}$, which itself is contained in a $\mathbb{P}^{1}$. Furthermore, the sign of $-K_{M}$ dictated by the shrinking soliton equation allows $M$ to contain only $(-1)$-curves, ruling out iterative blowups of $M_{\min}$ at a point.
These two properties limit the number of blowup points to one, leading to the statement of Theorem \ref{mainthm1}(i). The biholomorphism constructed between $M$ and
the manifolds of part (i) is torus-equivariant and uses the flow of $X$ and $JX$, hence naturally has the property regarding the vector field
stated in (ii). The toricity of the soliton metric follows from an application of the version of Matsushima's theorem for shrinking gradient K\"ahler-Ricci solitons proved in \cite{cds}. For this step, the assumption of bounded scalar curvature is crucial. Now knowing that $JX$
lies in the Lie algebra of the ambient torus means that it can be identified as it has the property that it minimises a certain functional, known as the
weighted volume functional \cite{cds, Tian-Zhu-II}. In fact, knowing the two possibilities for $M$ allows us to compute this vector field explicitly in each case. This yields (iii). Finally, knowing that the soliton is toric, the uniqueness statement of (iv) is immediate from \cite{charlie}.

\subsubsection{Application to the K\"ahler-Ricci flow}

For a complete shrinking gradient K\"ahler-Ricci soliton $(M,\,g,\,X)$ with $X=\nabla^{g}f$ for $f:M\rightarrow\mathbb{R}$ smooth, one can define an ancient solution $g(t),\,t<0,$ of the K\"ahler-Ricci flow
\begin{equation*}
\frac{\partial g(t)}{\partial t}=-\operatorname{Ric}(g(t))
\end{equation*}
 with $g(-1)=g$ by defining $g(t):=-t\varphi_{t}^{*}g, t<0,$  where $\varphi_{t}$ is a family of diffeomorphisms generated by the gradient vector field $-\frac{1}{t}X$ with $\varphi_{-1}=\operatorname{id}$, i.e.,
\begin{equation*}
\frac{\partial\varphi_{t}}{\partial t}(x)=-\frac{\nabla^g f(\varphi_{t}(x))}{2t},\qquad\varphi_{-1}=\operatorname{id.}
\end{equation*}
These K\"ahler-Ricci flows model the formation of finite time Type I singularities of the flow \cite{naber} which we now define.
We recall the following from \cite{topping} in the context of the K\"ahler-Ricci flow.

A family $(M,\,g(t))$ of smooth complete K\"ahler manifolds satisfying the K\"ahler-Ricci flow
\begin{equation*}
\frac{\partial g(t)}{\partial t}=-\operatorname{Ric}(g(t))
\end{equation*}
on a finite time interval $[0,\,T),\,T<+\infty$, is called a \emph{Type I K\"ahler-Ricci flow} if there exists a
constant $C > 0$ such that for all $t\in[0,\,T)$,
\begin{equation*}
\sup_{M}|\operatorname{Rm}_{g(t)}|_{g(t)}\leq\frac{C}{T-t}.
\end{equation*}
Such a solution is said to develop a \emph{Type I singularity} at time $T$
(and $T$ is called a \emph{Type I singular time}) if it cannot be smoothly extended past time
$T$. It is well known that this is the case if and only if
\begin{equation}\label{(1.3)}
\limsup_{t\to T^{-}}\sup_{M}|\operatorname{Rm}_{g(t)}|_{g(t)}=+\infty;
\end{equation}
see \cite{hamilton} for compact and \cite{shii} for complete flows. Here $\operatorname{Rm}_{g(t)}$ denotes the Riemannian curvature tensor of the metric $g(t)$.

Since Type I K\"ahler-Ricci flows $(M,\,g(t))$ have bounded curvature for each $t\in[0,T)$, the parabolic maximum principle applied to the evolution equation satisfied by $|\operatorname{Rm}|_{g(t)}^{2}$ shows that \eqref{(1.3)} is equivalent to
$$\sup_{M}|\operatorname{Rm}_{g(t)}|_{g(t)}\geq\frac{1}{8(T-t)}\qquad\textrm{for all $t\in[0,\,T)$}.$$
This motivates the following definition.

\begin{definition}[{\cite[Definition 1.2]{topping}}]\label{typeI}
Let $(M,\,g(t))$, $t\in [0,\,T)$, $T<+\infty,$ be a K\"ahler-Ricci flow.
A space-time sequence $(p_{i},\,t_{i})$ with $p_{i}\in M$ and $t_{i}\to T^{-}$
is called an \emph{essential blow-up sequence} if there exists a constant $c>0$ such that
$$|\operatorname{Rm}_{g(t_{i})}|_{g(t_{i})}\geq\frac{c}{T-t_{i}}.$$
A point $p\in M$ in a Type I K\"ahler-Ricci flow is called a \emph{Type I singular point} if there
exists an essential blow-up sequence with $p_{i}\to p$ on $M$. We denote the set of all Type I
singular points by $\Sigma_{I}$.
\end{definition}

The set $\Sigma_{I}$ has been characterised in \cite[Theorem 1.2]{topping}.
As already noted, in general it is known that a suitable blowup limit of a complete K\"ahler-Ricci flow at a point of $\Sigma_{I}$
is a non-flat shrinking gradient K\"ahler-Ricci soliton with bounded curvature \cite{topping, naber}.
Therefore, assuming the development of a finite time Type I singularity,
thanks to the classification given by Theorem \ref{mainthm1}, we are able to obtain as a corollary the following strong form of
the Feldman-Ilmanen-Knopf conjecture for such singularities on compact K\"ahler surfaces \cite[Example 2.2(3)]{FIK}.

\begin{mtheorem}[Non-collapsing]\label{thmc}
Let $(M,\,g(t))$ be a Type I K\"ahler-Ricci flow on $[0,\,T),\,T<+\infty,$ on a compact K\"ahler surface $M$ and
suppose that $x\in\Sigma_{I}$ is a Type I singular point as defined in Definition \ref{typeI}.
Then for every sequence $\lambda_{j}\to+\infty$, the rescaled K\"ahler-Ricci flows $(M,\,g_{j}(t),\,x)$ defined on $[-\lambda_{j}T,\,0)$ by $g_{j}(t):=\lambda_{j}g(T+\frac{t}{\lambda_{j}})$ subconverge in the smooth pointed Cheeger-Gromov topology to the unique shrinking gradient $U(2)$-invariant K\"ahler-Ricci soliton of Feldman-Ilmanen-Knopf \cite{FIK}
on the blowup of $\mathbb{C}^{2}$ at one point if and only if $\lim_{t\to T^{-}}\operatorname{vol}_{g(t)}(M)>0$.
\end{mtheorem}

This characterises the Feldman-Ilmanen-Knopf shrinking soliton as the unique shrinking soliton modelling finite time Type I non-collapsed singularities of the K\"ahler-Ricci flow on
compact K\"ahler surfaces. The ``if'' direction of Theorem \ref{thmc} is known to hold true for $U(n)$-invariant K\"ahler-Ricci flows on the blowup of $\mathbb{P}^{n}$ at one point \cite{guo}.
Moreover, on this manifold, it is known that any $U(n)$-invariant solution of the K\"ahler-Ricci flow developing a finite time singularity
is a singularity of Type I \cite{sing}. Similar results were obtained by M\'aximo \cite{maximo} for $n=2$.
However, contrary to a folklore conjecture, not every finite time singularity of the K\"ahler-Ricci flow is of Type I \cite{li1}, although this is expected to be the case for K\"ahler-Ricci flows on compact K\"ahler surfaces.

The proof of Theorem \ref{thmc} is by contradiction. Assuming volume non-collapsing, we consider the volume evolution of the unique $(-1)$-curve in the Feldman-Ilmanen-Knopf shrinking soliton under the K\"ahler-Ricci flow to rule out other possible shrinking solitons appearing as the rescaled limit. For the other direction, we assume volume collapsing and the
appearance of the Feldman-Ilmanen-Knopf shrinking soliton to derive a nonsensical lower bound on the volume of a $(-1)$-curve in the original manifold.
This direction crucially relies on the structure of collapsing singularities of the K\"ahler-Ricci flow in complex dimension $2$ given by \cite{tosatti10}
and the asymptotics and symmetry of the aforementioned soliton.

Given Theorem \ref{thmc}, we can now classify the finite time Type I rescaled limits of the K\"ahler-Ricci flow on a compact K\"ahler surface $M$. To this end, let
$(M,\,g(t))_{t\in[0,\,T)}$ be a K\"ahler-Ricci flow developing a finite Type I singularity when $t=T>0$. Take the blowup limit as is done in Theorem \ref{thmc}.
If $\lim_{t\to T^{-}}\operatorname{vol}_{g(t)}(M)>0$, then Theorem \ref{thmc} asserts that the blowup limit is the Feldman-Ilmanen-Knopf shrinking soliton on
the blowup of $\mathbb{C}^{2}$ at one point. This picture is consistent with finite time singularities of the K\"ahler-Ricci flow on
compact K\"ahler surfaces being of Type I. Indeed, under the assumption of non-collapsing, it is known that the flow contracts finitely many disjoint $(-1)$-curves on $M$
\cite[Theorem 3.8.3]{Bou-Eys-Gue}. On the other hand, if there is finite time collapsing at $t=T>0$, i.e., if
$\lim_{t\to T^{-}}\operatorname{vol}_{g(t)}(M)=0$, then either $\lim_{t\to T^{-}}\operatorname{diam}(M,\,g(t))=0$, which is a ``finite time extinction'',
or $\lim_{t\to T^{-}}\operatorname{diam}(M,\,g(t))>0$. In the former case, \cite{tosatti10} asserts that $M$ is Fano and
the K\"ahler class of $g(0)$ lies in $c_{1}(M)$. The work of
Perelman (see \cite{sesum1}) gives us the upper bound $\operatorname{diam}(M,\,g(t))\leq C(T-t)^{\frac{1}{2}}$, which,
for the re-scaled limit $g_{j}(t)$, $t<0$, translates to $\operatorname{diam}(M,\,g_{j}(t))\leq Ct$. This latter bound implies that
the rescaled limit is compact, hence being a shrinking soliton, is Fano with its (up to automorphism) unique shrinking soliton structure. In the latter case,
the blowup limit cannot be Fano as the compactness of such a manifold would imply that
$\lim_{t\to T^{-}}\operatorname{diam}(M,\,g(t))=0$, a contradiction. By Theorem \ref{thmc}, the blowup limit cannot be
the shrinking soliton of Feldman-Ilmanen-Knopf. Hence the only possibility is that the blowup limit is the cylinder
$\mathbb{C}\times\mathbb{P}^{1}$ or the hypothetical soliton of Conjecture \ref{conjecture}.
The precise soliton that appears would depend upon the proximity of the blowup point to a $(-1)$-curve.
This collapsing picture is also consistent with finite time singularities of the K\"ahler-Ricci flow on
compact K\"ahler surfaces being of Type I as under the assumption of finite time collapsing, it is known that the underlying complex manifold
is birational to a ruled surface \cite[Proposition 3.8.4]{Bou-Eys-Gue}.

\subsection{Outline of paper}

We begin in Section \ref{Jholo} by presenting the background material on $J$-holomorphic curves that we need to prove Theorem \ref{mainthm1}.
We then recall in Section \ref{intro2} the basics of shrinking Ricci and K\"ahler-Ricci solitons.
In Section \ref{pooly}, we digress and mention some basics
on polyhedrons and polyhedral cones that we need before moving on to some relevant information concerning Hamiltonian actions in Section \ref{hamilton}.
Section \ref{toric-geom} then comprises the background material on toric geometry that we need. In particular, we recall the definition of the
weighted volume functional and discuss its properties in Section \ref{weighted}. Moreover, in this section, we
determine  explicitly the unique holomorphic vector field on the manifolds of Theorem \ref{mainthm1}(i)
that could be the soliton vector field of a shrinking gradient K\"ahler-Ricci soliton with bounded scalar curvature.

In Section \ref{proof1}, we prove Theorem \ref{mainthm1}. We first prove in Proposition \ref{smoothclass} a smooth classification of the underlying manifold, a precursor
to the holomorphic classification given by Proposition \ref{holoo}. This section concludes by completing the proof of Theorem \ref{mainthm1}.

In the final section, namely Section \ref{proof2}, we prove Theorem \ref{thmc}.

\subsection{Acknowledgements}
The authors wish to thank Song Sun and Jeff Viaclovsky for useful discussions, Ovidiu Munteanu for providing the proof of Lemma \ref{batman},
and Max Hallgren for pointing out an oversight in the original proof of Theorem \ref{thmc}. The first author is supported by the grant Connect Talent ``COCOSYM’’ of the r\'{e}gion des Pays de la Loire, the second author is supported by NSF grant DMS-1906466, and the third author is supported by grants ANR-17-CE40-0034 of the French National Research Agency ANR (Project CCEM) and ANR-AAPG2020 (Project PARAPLUI).

\section{Preliminaries}\label{prelim}

\subsection{$J$-holomorphic curves}\label{Jholo}

In this section, we summarise the tools from the theory of $J$-holomorphic curves that we need in the context of K\"ahler manifolds. The source for this material is \cite{dusa2, dusa1}.

Let $(M,\,\,J)$ be an $n$-dimensional complex manifold and let $(\Sigma,\,j)$ be a compact Riemann surface with complex structures $J$ and $j$, respectively.
We say that a smooth map $u:\Sigma\to M$ is a \emph{$J$-holomorpic curve} if the differential $du$ is a complex linear map with respect to $j$ and $J$, i.e.,
$$J\circ du=du\circ j.$$ A smooth map $u:(\Sigma,\,j)\to(M,\,J)$ is $J$-holomorphic if and only if
\begin{equation}\label{holomorphic}
\bar{\partial}_{J}u=0,
\end{equation}
where $$\bar{\partial}_{J}u:=\frac{1}{2}(du+J\circ du\circ j).$$
By definition, a $J$-holomorphic curve is always parametrised. A $J$-holomorphic curve
$u:(\Sigma\,\,j)\to M$ is said to be \emph{multiply covered} if there exists a $J$-holomorphic
curve $u':(\Sigma',\,j')\to M$ and a branched covering $\phi:\Sigma\to\Sigma'$ of degree strictly greater than $1$
such that $u$ factors as $u=u'\circ\phi$. The curve $u$ is called \emph{simple} if it is not multiply covered. If $u$ is
a multiply covered $J$-holomorphic curve from $\mathbb{P}^{1}$, then by the Riemann-Hurwitz formula, $\Sigma'=\mathbb{P}^{1}$
also.

We henceforth restrict ourselves to $J$-holomorphic spheres, that is, when $\Sigma=\mathbb{P}^{1}$.
For a given homology class $A\in H_{2}(M,\,\mathbb{Z})$, we denote for such curves the moduli space of solutions to
\eqref{holomorphic} by
$$\mathcal{M}(A;\,J):=\{u\in C^{\infty}(\mathbb{P}^{1},\,M)\:|\:J\circ du=du\circ j,\,[u(\mathbb{P}^{1})]=A\}$$
and the subspace of simple solutions by
$$\mathcal{M}^{*}(A;\,J):=\{u\in\mathcal{M}(A;\,J)\:|\:\textrm{$u$ is simple}\}.$$

For a compact Riemannian manifold $N$, let $\Omega^{0}(N,\,E)$ denote the space of smooth sections of the bundle $E\to N$. Moreover, let $\Lambda^{0,\,1}:=\Lambda^{0,\,1}T^{*}\mathbb{P}^{1}$
denote the bundle of $1$-forms on $\mathbb{P}^{1}$ of type $(0,\,1)$. Assume now that $(M,\,J)$ is K\"ahler with a given K\"ahler form $\omega$
and for a given smooth (not necessarily $J$-holomorphic) curve $u:\mathbb{P}^{1}\to M$,
we define a map $$\mathcal{F}_{u}:\Omega^{0}(\mathbb{P}^{1},\,u^{*}TM)\to\Omega^{0}(\mathbb{P}^{1},\,\Lambda^{0,\,1}\otimes_{J}u^{*}TM)$$ as follows. Given $\xi\in\Omega^{0}(\mathbb{P}^{1},\,u^{*}TM)$,
let $$\Phi_{u}(\xi):u^{*}TM\to\exp_{u}(\xi)^{*}TM$$
denote the complex bundle isomorphism given by parallel transport along the geodesics $s\mapsto$\linebreak $\exp_{u(z)}(s\xi(z))$ with respect to the Levi-Civita connection
$\nabla$ induced by $\omega$. Then define
\begin{equation}\label{capitalf}
\mathcal{F}_{u}(\xi):=\Phi_{u}(\xi)^{-1}\bar{\partial}_{J}(\exp_{u}(\xi)).
\end{equation}
Write $\Omega_{J}^{0,\,1}(\mathbb{P}^{1},\,u^{*}TM):=\Omega^{0}(\mathbb{P}^{1},\,\Lambda^{0,\,1}\otimes_{J}u^{*}TM)$, where we drop the subscript $J$ when there is no ambiguity, and let $D_{u}$
denote the linearisation $d\mathcal{F}_{u}(0)$ of $\mathcal{F}_{u}$ at $0$. Then $D_{u}$ defines an operator
$$D_{u}:\Omega^{0}(\mathbb{P}^{1},\,u^{*}TM)\to\Omega_{J}^{0,\,1}(\mathbb{P}^{1},\,u^{*}TM),$$
which in our situation with $J$ complex is given by
\begin{equation}\label{derivative}
D_{u}\xi:=\frac{1}{2}\left(\nabla\xi+J(u)\nabla\xi\circ j\right)
\end{equation}
for every $\xi\in\Omega^{0}(\mathbb{P}^{1},\,u^{*}TM)$ \cite[Proposition 3.1.1]{dusa1}, i.e., $D_{u}\xi$ is the projection of $\nabla\xi$ onto $\Omega^{0,\,1}(\mathbb{P}^{1},\,u^{*}TM)$.
This is a real linear ``Cauchy-Riemann'' operator (cf.~\cite[Appendix C]{dusa1}), hence is Fredholm \cite[Theorem C.1.10]{dusa1}, meaning that it has closed range and finite dimensional kernel and
cokernel. The Riemann-Roch theorem asserts that its Fredholm index is
\begin{equation*}
\operatorname{index}\,D_{u}=2n+2c_{1}(u^{*}TM),
\end{equation*}
where $n=\dim_{\mathbb{R}}M$. In the case that $u:\mathbb{P}^{1}\to M$ is actually a $J$-holomorphic curve, $D_{u}$ is precisely the Dolbeault
$\bar{\partial}$-operator
$$\bar{\partial}:\Omega^{0}(u^{*}TM)\to\Omega^{0,\,1}(u^{*}TM).$$
If in addition $D_{u}$ is surjective, then $\mathcal{M}^{*}([u(\mathbb{P}^{1})];\,J)$ is a smooth oriented manifold near $u$
of real dimension $2n+2c_{1}(u^{*}TM)$ \cite[Theorem 3.1.5]{dusa1}.

The following is well-known.

\begin{prop}[Local deformations]\label{local1}
Let $M$ be a two-dimensional complex manifold with complex structure $J$, let $C$ be a simple embedded $J$-holomorphic sphere with
$C.C=0$, and let $D$ denote the open ball of radius $1$ in $\mathbb{C}$. Then there exists an open neighbourhood $U$ of $C$ that is diffeomorphic to $D\times\mathbb{P}^{1}$ with $\{t\}\times\mathbb{P}^{1}$ a $J$-holomorphic sphere in $M$ for each $t\in D$ and $\{0\}\times\mathbb{P}^{1}=C$.
\end{prop}

\begin{proof}
Fix a parametrisation $u:\mathbb{P}^{1}\to C\subset M$. As $u$ is $J$-holomorphic, we know that
$\bar{\partial}_{J}u=0$. The linearisation $D_{u}$ of $\mathcal{F}_{u}$ at $0$ is then Fredholm and is precisely the Dolbeault $\bar{\partial}$-operator with respect to $J$, namely
$$D_{u}=\bar{\partial}:\Omega^{0}(u^{*}TM)\to\Omega^{0,\,1}(u^{*}TM).$$ Moreover, as $C$ has trivial holomorphic normal bundle, we have the direct sum decomposition
$u^{*}TM=\mathcal{O}\oplus\mathcal{O}(2)$, a splitting that is respected by $\bar{\partial}$. Therefore, recalling the proof of \cite[Lemma 3.3.1]{dusa1}, we can consider the action of $\bar{\partial}$ on each factor separately. For any holomorphic line bundle $L\to\mathbb{P}^{1}$, the cokernel of
$\bar{\partial}:\Omega^{0}(\mathbb{P}^{1},\,L)\to\Omega^{0,\,1}(\mathbb{P}^{1},\,L)$ is precisely the Dolbeault cohomology group
$H^{0,\,1}_{\bar{\partial}}(\mathbb{P}^{1},\,L)$. Now, we have an isomorphism $$H^{0,\,1}_{\bar{\partial}}(\mathbb{P}^{1},\,L)\cong(H^{1,\,0}_{\bar{\partial}}(\mathbb{P}^{1},\,L^{*}))^{*},$$
where $H^{1,\,0}_{\bar{\partial}}(\mathbb{P}^{1},\,L^{*})$ is the space of holomorphic one-forms with values in the dual bundle $L^{*}$ and which itself is isomorphic to  $H^{0}(\mathbb{P}^{1},\,L^{*}\otimes K_{\mathbb{P}^{1}})$, the space of holomorphic sections of the bundle
$L^{*}\otimes K_{\mathbb{P}^{1}}$ by Kodaira-Serre duality. Hence
$H^{0,\,1}_{\bar{\partial}}(\mathbb{P}^{1},\,\mathcal{O})=H^{0,\,1}_{\bar{\partial}}(\mathbb{P}^{1},\,\mathcal{O}(2))=0$. In particular,
$D_{u}$ is surjective of Fredholm index $8$ so that $\mathcal{M}^{*}([C];\,J)$ is a smooth oriented manifold of real dimension $8$ near $u$. Indeed,
it follows from \cite[Corollary 3.3.4]{dusa1} that $D_{v}$ is surjective for every $v\in\mathcal{M}^{*}([C];\,J)$, hence
$\mathcal{M}^{*}([C];\,J)$ itself is a smooth oriented manifold of real dimension $8$.

Recall that $\mathcal{M}^{*}([C];\,J)$ comprises parametrised $J$-holomorphic curves.
The six-dimensional real Lie group $PSL(2,\,\mathbb{C})$, which we henceforth denote by $G$, acts freely on $\mathcal{M}^{*}([C];\,J)$ via reparametrisation:
$$g\cdot v=v\circ g^{-1}\quad\textrm{for all $g\in G$ and $v\in\mathcal{M}^{*}([C];\,J)$}.$$
We consider the quotient space
$$\widetilde{\mathcal{M}}^{*}([C];\,J):=\mathcal{M}^{*}([C];\,J)/G.$$
This is precisely the space of $J$-holomorphic spheres in $M$ in the same homology class as $C$
and is a smooth oriented manifold of real dimension $8-6=2$. As $C$ is simple and embedded, McDuff's adjunction formula \cite[Corollary E.1.7]{dusa1} implies that
every sphere in $\widetilde{\mathcal{M}}^{*}([C];\,J)$ is embedded in $M$. In addition, the fact that $C.C=0$ implies that any two distinct $\mathbb{P}^{1}$'s in $\widetilde{\mathcal{M}}^{*}([C];\,J)$ are disjoint in $M$.

Set $\mathcal{M}^{*}([C];\,J)\times_{G}\mathbb{P}^{1}\equiv(\mathcal{M}^{*}([C];\,J)\times\mathbb{P}^{1})/G,$
where $G$ acts on $\mathcal{M}^{*}([C];\,J)\times\mathbb{P}^{1}$ by $g\cdot(v,\,z)\mapsto(v\circ g^{-1},\,g\cdot z)$.
Then $\mathcal{M}^{*}([C];\,J)\times_{G}\mathbb{P}^{1}$ is a smooth manifold of real dimension $4$ which is a $\mathbb{P}^{1}$-bundle
over $\widetilde{\mathcal{M}}^{*}([C];\,J)$. We define an evaluation map $\operatorname{ev}$ by
$$\operatorname{ev}:\mathcal{M}^{*}([C];\,J)\times_{G}\mathbb{P}^{1}\mapsto M,\qquad[(v,\,z)]\mapsto v(z).$$
This is a smooth map between two oriented smooth manifolds of the same dimension that maps every fibre
$\{[(v,\,z)]\:|\:z\in\mathbb{P}^{1}\}$ biholomorphically onto an embedded $J$-holomorphic sphere in $M$, with distinct fibres
being mapped to distinct $J$-holomorphic spheres in $M$ with $\{[(u,\,z)]\:|\:z\in\mathbb{P}^{1}\}$ being mapped to $C$.
In particular, $\operatorname{ev}$ is an immersion between two manifolds of the same dimension, hence is a local diffeomorphism.
Choosing a trivialisation of the $\mathbb{P}^{1}$-bundle in a neighbourhood of
the fibre $\{[(u,\,z)]\:|\:z\in\mathbb{P}^{1}\}$ now yields the result.
\end{proof}

Next, for a compact Riemannian manifold $N$ and for $k\geq1$ an integer and $p>2$ a real number, let $W^{k,\,p}(N,\,E)$ denote the completion of the space $\Omega^{0}(N,\,E)$
of smooth sections of the bundle $E\to N$ with respect to the Sobolev $W^{k,\,p}$-norm. Again, assume that $(M,\,J)$ is K\"ahler with K\"ahler form $\omega$ and endow $(\mathbb{P}^{1},\,j)$ with the Fubini-Study form $\omega_{\mathbb{P}^{1}}$ compatible with $j$. For a given smooth curve $u:\mathbb{P}^{1}\to M$ and real number $p>2$,
let
\begin{equation}\label{spaces}
\mathcal{X}_{u}^{p}:=W^{1,\,p}(\mathbb{P}^{1},\,u^{*}TM),\qquad\mathcal{Y}_{u}^{p}:=L^{p}(\mathbb{P}^{1},\,\Lambda^{0,\,1}\otimes_{J}u^{*}TM),
\end{equation}
where all relevant norms are understood to be with respect to $\omega$ and $\omega_{\mathbb{P}^{1}}$
and the Levi-Civita connection $\nabla$ determined by $\omega$. Then the maps $\mathcal{F}_{u}$ and $D_{u}$
defined above for smooth sections extend in a natural way to maps
$\mathcal{F}_{u}:\mathcal{X}_{u}^{p}\to\mathcal{Y}_{u}^{p}$.

One can prove that if $u$ is an \emph{approximate} $J$-holomorphic curve with \emph{sufficiently surjective} operator $D_{u}$, then there are $J$-holomorphic curves
near $u$ and the moduli space can be modelled on a neighbourhood of zero in the kernel of $D_{u}$. More precisely, we have the following theorem.

\begin{theorem}[{\cite[Theorem 3.3.4]{dusa2} with $\Sigma=\mathbb{P}^{1}$} and $u:\Sigma\to M$ smooth]\label{beauty}
Let $p>2$ and let $\norm{\cdot}$ denote the operator norm. Then for every constant $c_{0}>0$, there exist constants $\delta>0$ and $c>0$ such that the following holds.
Let $u:\mathbb{P}^{1}\to M$ be a smooth map and $Q_{u}:\mathcal{Y}_{u}^{p}\to\mathcal{X}_{u}^{p}$ be a right inverse of $D_{u}$ such that
$$\norm{Q_{u}}\leq c_{0},\qquad\norm{du}_{L^{p}}\leq c_{0},\qquad\norm{\bar{\partial}_{J}u}_{L^{p}}\leq\delta,$$
with respect to a metric on $\mathbb{P}^{1}$ such that $\operatorname{vol}(\mathbb{P}^{1})\leq c_{0}$. Then for every $\xi\in\ker(D_{u})$
with $\norm{\xi}_{L^{p}}\leq\delta$, there exists a section $\tilde{\xi}=Q_{u}\eta\in\mathcal{X}_{u}^{p}$ such that
$$\bar{\partial}_{J}(\exp_{u}(\xi+Q_{u}\eta))=0,\qquad\norm{Q_{u}\eta}_{W^{1,\,p}}\leq c\norm{\bar{\partial}_{J}(\exp_{u}(\xi))}_{L^{p}}.$$
\end{theorem}
This theorem is proved using the implicit function theorem. Given a surjective operator $D_{u}$, one technique for constructing a right inverse $Q_{u}$ is
to reduce the domain of $D_{u}$ by imposing pointwise conditions on $\xi$ so that the resulting operator is bijective, and then taking $Q_{u}$ to be the inverse of this
restricted operator. We will use this to prove the following two corollaries of this theorem.

\begin{corollary}[Deformation of trivially-embedded curves]\label{curve}
Let $M$ be a manifold of real dimension $4$, let $(g,\,J)$ and $(\tilde{g},\,\widetilde{J})$ be two K\"ahler structures on $M$, let $p>2$, and
let $u:(\mathbb{P}^{1},\,j)\to(M,\,\widetilde{J})$ be a smooth $\widetilde{J}$-holomorphic curve with trivial self-intersection.
Denote the Levi-Civita connection of $\tilde{g}$ by $\widetilde{\nabla}$. Then
for all $x\in u(\mathbb{P}^{1})$, there exists $\varepsilon>0$ such that if
\begin{equation}\label{difference1}
|g-\tilde{g}|_{\tilde{g}}+|\widetilde{\nabla}(g-\tilde{g})|_{\tilde{g}}+|J-\widetilde{J}|_{\tilde{g}}<\varepsilon
\end{equation}
on some sufficiently large compact subset $K\subset M$ containing $u(\mathbb{P}^{1})$, then there exists a unique
smooth section $\tilde{\xi}\in\Gamma(u^{*}TM)$ with $\tilde{\xi}(x)=0$ and
$\norm{\tilde{\xi}}_{C^{0}}\leq C\norm{J-\widetilde{J}}_{C^{0}(\mathbb{P}^{1},\,\tilde{g})}$ such that $v:=\exp_{u}(\tilde{\xi}):(\mathbb{P}^{1},\,j)\to(M,\,J)$
is a smooth $J$-holomorphic curve (in the same homology class as $u(\mathbb{P}^{1})$ with $x\in v(\mathbb{P}^{1})$).
\end{corollary}

\begin{proof}
Let $\widetilde{F}_{u}$ denote the map \eqref{capitalf} corresponding to the data $(u,\,\tilde{g},\,\widetilde{J})$
and recall from the proof of Proposition \ref{local1} that the
linearisation $\widetilde{D}_{u}$ of $\widetilde{F}_{u}$ at $0$ with respect to $\widetilde{J}$ is
Fredholm of index $8$ and is precisely the Dolbeault $\bar{\partial}$-operator with respect to $\widetilde{J}$, namely
$$\widetilde{D}_{u}=\bar{\partial}:\Omega^{0}(u^{*}TM)\to\Omega_{\widetilde{J}}^{0,\,1}(u^{*}TM).$$
Via the direct sum decomposition $u^{*}TM=\mathcal{O}\oplus\mathcal{O}(2)$,
the kernel of $\widetilde{D}_{u}$ is spanned by $\{1,\,z_{1}^{2},\,z_{1}z_{2},\,z_{2}^{2}\}$
with $[z_{1}:z_{2}]$ homogeneous coordinates on $\mathbb{P}^{1}$.
Identifying $x$ with its pre-image under $u$, restrict $\widetilde{D}_{u}$ to the subspace $\Omega^{0}(u^{*}TM)_{(0)}$ of $\Omega^{0}(u^{*}TM)$ of smooth sections that vanish in the tangential directions at the points $x$, $z_{1}=0$, and $z_{2}=0$ on $\mathbb{P}^{1}$, and vanish in the normal direction at $x$. (If $z_{i}(x)=0$ for some $i=1,\,2$, then just choose an arbitrary point on $\mathbb{P}^{1}$ distinct from $z_{1}=0$ and $z_{2}=0$ for the sections to vanish.) Then the restriction $$\widetilde{D}^{(0)}_{u}:\Omega^{0}(u^{*}TM)_{(0)}\to\Omega_{\widetilde{J}}^{0,\,1}(u^{*}TM)$$ is an isomorphism. Fix $p>2$ and let $(\widetilde{\mathcal{X}}_{u}^{p})_{(0)}$ and $\widetilde{\mathcal{Y}}_{u}^{p}$ denote
the Sobolev completion of $\Omega^{0}(u^{*}TM)_{(0)}$ with respect to the $W^{1,\,p}$-norm
and the completion of $\Omega^{0}(\mathbb{P}^{1},\,\Lambda^{0,\,1}\otimes_{\tilde{J}}u^{*}TM)$
with respect to the $L^{p}$-norm induced by $\tilde{g}$ and the choice of K\"ahler metric on $\mathbb{P}^{1}$, respectively.
Then $\widetilde{D}^{(0)}_{u}$ defines an isomorphism $\widetilde{D}^{(0)}_{u}:(\widetilde{\mathcal{X}}_{u}^{p})_{(0)}\to\widetilde{\mathcal{Y}}_{u}^{p}$.

Next consider $\bar{\partial}_{J}u$.
Let $\mathcal{X}_{u}^{p}$ and $\mathcal{Y}_{u}^{p}$ be as in \eqref{spaces} defined with respect to $\omega$ and the choice of K\"ahler metric on $\mathbb{P}^{1}$
and let $(\mathcal{X}_{u}^{p})_{(0)}$ denote the Sobolev completion of $\Omega^{0}(u^{*}TM)_{(0)}$ with respect to the $W^{1,\,p}$-norm
induced by the aforementioned metrics. Then the linearisation $D_{u}$ defines a map
$$D_{u}:\mathcal{X}_{u}^{p}\to\mathcal{Y}_{u}^{p}$$
which we can restrict to $(\mathcal{X}_{u}^{p})_{(0)}$ and compose with the projection $\operatorname{pr}:\mathcal{Y}_{u}^{p}\to\widetilde{\mathcal{Y}}_{u}^{p}$
to obtain a map
$$(\operatorname{pr}\circ D_{u})^{(0)}:(\mathcal{X}_{u}^{p})_{(0)}\to\widetilde{\mathcal{Y}}_{u}^{p}.$$
Explicitly, the composition $\operatorname{pr}\circ D_{u}$ is given by
\begin{equation}\label{finished}
(\operatorname{pr}\circ D_{u})(\xi)=\frac{1}{2}\left(\nabla\xi+\widetilde{J}\nabla\xi\circ j\right).
\end{equation}
As clearly $(\mathcal{X}_{u}^{p})_{(0)}=(\widetilde{\mathcal{X}}_{u}^{p})_{(0)}$, we also have an isomorphism
$\widetilde{D}^{(0)}_{u}:(\mathcal{X}_{u}^{p})_{(0)}\to\widetilde{\mathcal{Y}}_{u}^{p}$. Thus,
from the openness of the invertibility of bounded linear operators, we know that there exists $\delta>0$ such that
$\norm{\widetilde{D}^{(0)}_{u}-(\operatorname{pr}\circ D_{u})^{(0)}}<\delta$
implies the invertibility of $(\operatorname{pr}\circ D_{u})^{(0)}$. In light of \eqref{derivative} and \eqref{finished},
we estimate that $$\norm{\widetilde{D}_{u}-\operatorname{pr}\circ D_{u}}\leq C\norm{\nabla-\widetilde{\nabla}}_{C^{0}(\mathbb{P}^{1},\,\tilde{g})},$$
and so $(\operatorname{pr}\circ D_{u})^{(0)}$ is invertible if \eqref{difference1}
holds true for $\varepsilon>0$ sufficiently small.
Moreover, if \linebreak$\norm{J-\widetilde{J}}_{C^{0}(\mathbb{P}^{1},\,\tilde{g})}$ is sufficiently small, then
$\operatorname{pr}$ is an isomorphism. Hence, by shrinking $\varepsilon>0$ further if necessary, we can assert that
the restricted map
$$D_{u}^{(0)}:(\mathcal{X}_{u}^{p})_{(0)}\to\mathcal{Y}_{u}^{p}$$
is itself an isomorphism. As
$$\norm{\bar{\partial}_{J}u}_{L^{p}}\leq C\biggl(\norm{\bar{\partial}_{J}u-\bar{\partial}_{\widetilde{J}}u}_{L^{p}}+\underbrace{\norm{\bar{\partial}_{\widetilde{J}}u}_{L^{p}}}_{=\,0}\biggr)\leq C\norm{J-\widetilde{J}}_{C^{0}(\mathbb{P}^{1},\,\tilde{g})},$$
control on $\|J-\widetilde{J}\|_{C^{0}(\mathbb{P}^{1},\,\tilde{g})}$ allows us to
assume that $\norm{\bar{\partial}_{J}u}_{L^{p}}$ is as small as we please. Therefore, applying
Theorem \ref{beauty} with $\xi=0$, we deduce that for all $\varepsilon>0$ sufficiently small, there exists a unique section
$\tilde{\xi}\in(\mathcal{X}_{u}^{p})_{(0)}$ such that the map
$v:=\exp_{u}(\tilde{\xi})$ is $J$-holomorphic and
$\norm{\tilde{\xi}}_{W^{1,\,p}}\leq C\norm{\bar{\partial}_{J}u}_{L^{p}}$. Thus,
\begin{equation*}
\norm{\tilde{\xi}}_{W^{1,\,p}}\leq C\norm{\bar{\partial}_{J}u}_{L^{p}}\leq C\norm{J-\widetilde{J}}_{C^{0}(\mathbb{P}^{1},\,\tilde{g})}.
\end{equation*}
The desired estimate on $\tilde{\xi}$ now follows from Sobolev embedding.
The fact that $v$ is smooth follows from elliptic regularity and the smoothness of $J$
\cite[Proposition 3.1.9]{dusa1}. By construction, $\tilde{\xi}(x)=0$ so that $x\in v(\mathbb{P}^{1})$,
and $v(\mathbb{P}^{1})$ lies in the same homology class as $u(\mathbb{P}^{1})$, hence $v$ has the required properties.
Finally, uniqueness of $v$ is a consequence of the triviality of the normal bundle of $v(\mathbb{P}^{1})$ and the positivity of intersections of complex subvarieties
in a complex surface \cite[Theorem 2.6.3]{dusa1}.
\end{proof}

The next corollary is reminiscent of \cite[Theorem 5]{kodaira1}.

\begin{corollary}[Deformation of $(-1)$-curves]\label{curve2}
Let $M$ be a manifold of real dimension $4$, let $(g,\,J)$ and $(\tilde{g},\,\widetilde{J})$ be two K\"ahler structures on $M$, let $p>2$, and
let $u:(\mathbb{P}^{1},\,j)\to(M,\,\widetilde{J})$ be a smooth $\widetilde{J}$-holomorphic $(-1)$-curve.
Denote the Levi-Civita connection of $\tilde{g}$ by $\widetilde{\nabla}$. Then there exists $\varepsilon>0$ such that if
\begin{equation*}
|g-\tilde{g}|_{\tilde{g}}+|\widetilde{\nabla}(g-\tilde{g})|_{\tilde{g}}+|J-\widetilde{J}|_{\tilde{g}}<\varepsilon
\end{equation*}
on some sufficiently large compact subset $K\subset M$ containing $u(\mathbb{P}^{1})$, then there exists a unique
smooth section $\tilde{\xi}\in\Gamma(u^{*}TM)$ with $\norm{\tilde{\xi}}_{C^{0}}\leq C\norm{J-\widetilde{J}}_{C^{0}(\mathbb{P}^{1},\,\tilde{g})}$ such that $v:=\exp_{u}(\tilde{\xi}):(\mathbb{P}^{1},\,j)\to(M,\,J)$ is a smooth $J$-holomorphic $(-1)$-curve (in the same homology class as $u(\mathbb{P}^{1})$).
\end{corollary}

\begin{proof}
Let $\widetilde{F}_{u}$ denote the map \eqref{capitalf} corresponding to the data $(u,\,\tilde{g},\,\widetilde{J})$
and recall from the proof of Proposition \ref{local1} that the
linearisation $\widetilde{D}_{u}$ of $\widetilde{F}_{u}$ at $0$ with respect to $\widetilde{J}$ is
precisely the Dolbeault $\bar{\partial}$-operator with respect to $\widetilde{J}$, namely
$$\widetilde{D}_{u}=\bar{\partial}:\Omega^{0}(u^{*}TM)\to\Omega_{\widetilde{J}}^{0,\,1}(u^{*}TM).$$
This is Fredholm of index $6$ (cf.~the proof of Proposition \ref{local1}) and
via the direct sum decomposition $u^{*}TM=\mathcal{O}(-1)\oplus\mathcal{O}(2)$,
the kernel of $\widetilde{D}_{u}$ is spanned by $\{z_{1}^{2},\,z_{1}z_{2},\,z_{2}^{2}\}$
with $[z_{1}:z_{2}]$ homogeneous coordinates on $\mathbb{P}^{1}$.
Restrict $\widetilde{D}_{u}$ to the subspace $\Omega^{0}(u^{*}TM)_{(1)}$ of $\Omega^{0}(u^{*}TM)$ of smooth sections that vanish in the
tangential directions at the points $z_{1}=0$, $z_{2}=0$, and at an arbitrary point of $\mathbb{P}^{1}$
distinct from $z_{1}=0$ and $z_{2}=0$. Then the restriction $$\widetilde{D}^{(1)}_{u}:\Omega^{0}(u^{*}TM)_{(1)}\to\Omega_{\widetilde{J}}^{0,\,1}(u^{*}TM)$$ defines an isomorphism.
The proof now proceeds verbatim as that of Corollary \ref{curve} without the last sentence.
\end{proof}

\subsection{Shrinking Ricci solitons}\label{intro2}

The metrics we are interested in are the following.
\begin{definition}
A \emph{shrinking Ricci soliton} is a triple $(M,\,g,\,X)$, where $M$ is a Riemannian manifold endowed with a complete Riemannian metric $g$
and a vector field $X$ satisfying the equation
\begin{equation}\label{soliton1}
\Ric(g)+\frac{1}{2}\mathcal{L}_{X}g=\frac{1}{2}g.
\end{equation}
We call $X$ the \emph{soliton vector field} and say that $(M,\,g,\,X)$
is a \emph{gradient} Ricci soliton if $X=\nabla^{g} f$ for some real-valued smooth function $f$ on $M$.
In this latter case, equation \eqref{soliton1} reduces to
\begin{equation*}
\Ric(g)+\operatorname{Hess}_{g}(f)=\frac{1}{2}g,
\end{equation*}
where $\operatorname{Hess}_{g}$ denotes the Hessian with respect to $g$.

If $g$ is complete and K\"ahler with K\"ahler form $\omega$, then we say that $(M,\,g,\,X)$ is a \emph{shrinking gradient K\"ahler-Ricci soliton} if
$X=\nabla^{g} f$ for some real-valued smooth function $f$ on $M$, $X$ is complete and real holomorphic, and
\begin{equation}\label{krseqn}
\rho_{\omega}+i\partial\bar{\partial}f=\omega,
\end{equation}
where $\rho_{\omega}$ is the Ricci form of $\omega$. For gradient Ricci solitons and
gradient K\"ahler-Ricci solitons, the function $f$ satisfying $X=\nabla^g f$ is called the \emph{soliton potential}.
\end{definition}

As the next result shows, the soliton potential of a complete non-compact shrinking gradient Ricci soliton grows quadratically with respect to the distance.

\begin{theorem}[{\cite[Theorem 1.1]{caoo}}]\label{theo-basic-prop-shrink}
Let $(M,\,g,\,X)$ be a complete non-compact shrinking gradient Ricci soliton with soliton vector field $X=\nabla^{g}f$ for a smooth real-valued function $f:M\to\mathbb{R}$.
Then for $x\in M$, $f$ satisfies the estimates
$$\frac{1}{4}(d_{g}(p,\,x)-c_{1})^{2}-C\leq f(x)\leq\frac{1}{4}(d_{g}(p,\,x)+c_{2})^{2}$$
for some $C>0$, where $d_{g}(p,\,\cdot)$ denotes the distance to a fixed point $p\in M$ with respect to $g$.
Here, $c_{1}$ and $c_{2}$ are positive constants depending only on the real dimension of $M$ and
the geometry of $g$ on the unit ball $B_{p}(1)$ based at $p$.
\end{theorem}

\noindent In particular, $f$ is proper.

We also know the following regarding the asymptotics of four-dimensional shrinking gradient Ricci solitons.

\begin{lemma}\label{batman}
Let $(M,\,g,\,X)$ be a complete non-compact shrinking gradient Ricci soliton of real dimension $4$ with soliton vector field $X=\nabla^{g}f$ for a smooth real-valued function $f:M\to\mathbb{R}$
and with bounded scalar curvature $\RR_{g}$ such that $\RR_{g}\to0$ along every integral curve of $X$. Then
$\RR_{g}\to0$. Moreover, there exists a constant $C>0$ such that $0\leq\RR_{g}\leq Cf^{-1}$ outside a sufficiently large compact subset of $M$.
\end{lemma}

\begin{proof}
On a shrinking gradient Ricci soliton of real dimension $4$ with bounded scalar curvature $\RR_g$, we see from \cite[Theorem 1.3]{wang22} that the bounds \cite[equation (3.4)]{wang22} hold true so that \cite[Theorem 3.1]{wang22} applies. The Harnack estimate from \cite[equation (3.73)]{wang22} then implies that if $\RR_g$ is strictly
smaller than the constant in this Harnack estimate at some point $x$ in the level set $\{f=t_1\}$ for $t_1\in\mathbb{R}$ with $\{X=0\}\subset f^{-1}((-\infty,\,t_{1}])$, then $\RR_g$ decays like $Cf^{-1}$ along the integral curve passing through $x$ for some constant $C>0$ independent of $x$. Thus for the first assertion, it suffices to show that there exists $t_1\in\mathbb{R}$ with $\{X=0\}\subset f^{-1}((-\infty,\,t_{1}])$ so that $\RR_g$ is as small as we please on $\{f=t_1\}$.

To this end, note that since $\RR_g$ is bounded, the zero set of $X$ is compact (cf.~\cite[Proof of Lemma 2.26]{cds}), hence by properness of $f$ (cf.~Theorem \ref{theo-basic-prop-shrink}),
there exists $t_0>0$ so that $\{X=0\}\subset f^{-1}((-\infty,\,\frac{t_{0}}{2}])$. Through the gradient flow of $f$, the level sets $\{f=t\}$ are therefore diffeomorphic to $\{f=t_0\}$ for all $t>t_0$. In particular, all integral curves of $X$ may be parametrised by $\{f=t_0\}$. Let $x_{0}\in\{f=t_0\}$ and choose $\varepsilon>0$. Then since $\RR_{g}\to0$ along each integral curve of $X$ by assumption and $\RR_{g}\geq0$ \cite{Zhang-Com-Ricci-Sol}, there exists $x'_{0}$ lying along the integral curve of $X$ passing through $x_{0}$ with $f(x'_{0}):=t'_{0}>t_{0}$ such that $0\leq\RR_g(x'_{0}) <\varepsilon$. We can then find an open neighbourhood of $x'_{0}$ in $\{f=t'_{0}\}$ such that $0\leq\RR_g<2\varepsilon$. Flowing this neighbourhood back to $x_0$ along $-X$, we obtain an open neighbourhood $U_{0}$ of $x_0$ in $\{f=t_0\}$. By properness of $f$, the level set $\{f=t_0\}$ is compact and so can be covered by finitely many such neighbourhoods $U_{i},i=0,\ldots,N$. Letting $t_{1}$ denote the maximum of the corresponding $t'_{i},i=0,\ldots,N$, we find that $0\leq\RR_g<2\varepsilon$ on $\{f=t_1\}$ and $\{X=0\}\subset f^{-1}((-\infty,\,t_{1}])$, as required. By \cite[equation (3.73)]{wang22}, it now follows that $\RR_g$ decays globally like $Cf^{-1}$.
\end{proof}

Complex two-dimensional complete non-compact shrinking gradient K\"ahler-Ricci solitons with scalar curvature tending to zero at infinity were classified in
\cite[Theorem E(3)]{cds}. They comprise the flat Gaussian soliton on $\mathbb{C}^{2}$ and the example of Feldman-Ilmanen-Knopf \cite{FIK} on the blowup of $\mathbb{C}^{2}$ at one point, up to the action of $GL(2,\,\mathbb{C})$.

\subsection{Polyhedrons and polyhedral cones}\label{pooly}

We take the following from \cite{cox} and \cite[Appendix A]{wu}.

Let $E$ be a real vector space of dimension $n$ and let $E^{*}$ denote the dual. Write $\langle\cdot\,,\,\cdot\rangle$ for the evaluation $E^{*}\times E\to\mathbb{R}$. Furthermore, assume that we are given a \emph{lattice} $\Gamma \subset E$, that is, an additive subgroup $\Gamma \simeq \Z^n$. This gives rise to a dual lattice $\Gamma^* \subset E^*$. For any $\nu\in E$, $c\in\mathbb{R}$, let
$K(\nu,\,c)$ be the (closed) half space $\{x\in E\:|\:\langle\nu,\,x\rangle\geq c\}$ in $E$. Then we have:

\begin{definition}
A \emph{polyhedron} $P$ in $E$ is a finite intersection of half spaces,
i.e., $$P=\bigcap_{i=1}^{r}K(\nu_{i},\,c_{i})\qquad\textrm{for $\nu_{i}\in E^{*},\,c_{i}\in\mathbb{R}$}.$$
It is called a \emph{polyhedral cone} if all $c_{i}=0$, and moreover a \emph{rational polyhedral cone} if all $\nu_i \in \Gamma^*$ and $c_i = 0$. In addition, a polyhedron is called \emph{strongly convex} if it does not contain any affine subspace of $E$.
\end{definition}

The following definition will be useful.

\begin{definition}
A polyhedron $P \subset E^{*}$ is called \emph{Delzant} if its set of vertices is non-empty and each vertex $v \in P$ has the property that there are precisely $n$ edges $\{e_1, \dots e_n\}$ (one-dimensional faces) emanating from $v$ and there exists a basis $\{\varepsilon_1, \dots, \varepsilon_n\}$ of $\Gamma^*$ such that $\varepsilon_i$ lies along the ray $\R (e_i - v) $.
\end{definition}

\noindent Note that any such $P$ is necessarily strongly convex.

The asymptotic cone of a polyhedron contains all the directions going off to infinity in the polyhedron.

\begin{definition}\label{recession}
Let $P$ be a polyhedron in $E$. Its \emph{asymptotic cone}, denoted
by $\mathcal{C}(P)$, is the set of vectors $\alpha\in E$ with the property that there exists $\alpha^{0}\in E$
such that $\alpha^{0}+t\alpha\in P$ for sufficiently large $t>0$.
\end{definition}

The asymptotic cone may be identified as follows.

\begin{lemma}[{\cite[Lemma A.3]{wu}}]\label{alpha}
If $P=\bigcap_{i=1}^{r}K(\nu_{i},\,c_{i})$, then $\mathcal{C}(P)=\bigcap_{i=1}^{r}K(\nu_{i},\,0)$.
\end{lemma}
In particular, the asymptotic cone of a polyhedron is a polyhedral cone. In addition, we see that
for two polyhedrons $P,\,Q,$ in $E$, $$Q\subseteq P\implies\mathcal{C}(P)\subseteq\mathcal{C}(Q).$$

Compact polyhedrons can be characterised by their asymptotic cone.

\begin{lemma}[{\cite[Corollary A.9]{wu}}]\label{beta}
A polyhedron $P$ is compact if and only if $\mathcal{C}(P)=\{0\}$.
\end{lemma}

We also have:
\begin{definition}
The \emph{dual} of a polyhedral cone $C$ is the set $C^{\vee}=\{x\in E^{*}\:|\:\langle x,\,C\rangle\geq0\}$.
\end{definition}

It is clear that for two polyhedrons $P,\,Q,$ in $E$, $$Q\subseteq P\implies P^{\vee}\subseteq Q^{\vee}.$$

\subsection{Hamiltonian actions}\label{hamilton}

Recall what it means for an action to be Hamiltonian.

\begin{definition}
Let $(M,\,\omega)$ be a symplectic manifold and let $T$ be a real torus acting by symplectomorphisms on $(M,\,\omega)$.
Denote by $\mathfrak{t}$ the Lie algebra of $T$ and by $\mathfrak{t}^{*}$ its dual. Then we say that the action of $T$ is \emph{Hamiltonian}
if there exists a smooth map $\mu_{\omega}:M\to\mathfrak{t}^{*}$ such that for all $\zeta\in\mathfrak{t}$,
\begin{equation*}
-\omega\lrcorner\zeta=du_{\zeta},
\end{equation*}
where $u_{\zeta}(x)=\langle\mu_{\omega}(x),\,\zeta\rangle$ for all $\zeta\in\mathfrak{t}$ and $x\in M$
and $\langle\cdot\,,\cdot\rangle$ denotes the dual pairing between $\mathfrak{t}$ and $\mathfrak{t}^{*}$.
We call $\mu_{\omega}$ the \emph{moment map} of the $T$-action and we call $u_{\zeta}$ the \emph{Hamiltonian (potential)} of $\zeta$.
\end{definition}

Define
\begin{equation*}
\Lambda_{\omega}:=\{Y\in\mathfrak{t}\:|\: \textrm{$\mu_{\omega}(Y)$ is proper and bounded below}\}\subseteq\mathfrak{t}.
\end{equation*}
By Theorem \ref{theo-basic-prop-shrink}, this set is non-empty for $\omega$ a complete non-compact shrinking gradient K\"ahler-Ricci soliton. In addition,
it can be identified through the image of $\mu_{\omega}$ in the following way.

\begin{prop}[{\cite[Proposition 1.4]{wu}}]\label{identify}
Let $(M,\,\omega)$ be a (possibly non-compact) symplectic manifold of real dimension $2n$ with symplectic form $\omega$ on which there is a Hamiltonian
action of a real torus $T$ with moment map $\mu_{\omega}:M\to\mathfrak{t}^{*}$, where $\mathfrak{t}$ is the Lie algebra of $T$ and $\mathfrak{t}^{*}$
its dual. Assume that the fixed point set of $T$ is compact and that $\Lambda_{\omega}\neq\emptyset$. Then $\Lambda_{\omega}=\operatorname{int}(\mathcal{C}(\mu_{\omega}(M))^{\vee})$.
\end{prop}

\subsection{Toric geometry}\label{toric-geom}

In this section, we collect together some standard facts from toric geometry as well as recall those results from \cite{charlie} that we require.
We begin with the following definition.
\begin{definition}\label{toricmanifold}
A \emph{toric manifold} is an $n$-dimensional complex manifold $M$ endowed
with an effective holomorphic action of the algebraic torus $\Cstarn$ such that the following hold true.
\begin{itemize}
  \item The fixed point set of the $\Cstarn$-action is compact.
  \item There exists a point $p\in M$ with the property that the orbit $\Cstarn \cdot p \subset M$ forms a dense open subset of $M$.
\end{itemize}
\end{definition}
We will often denote the dense orbit simply by $\Cstarn \subset M$ in what follows.
The $\Cstarn$-action of course determines the action of the real torus $T^n \subset \Cstarn$.

\subsubsection{Divisors on toric varieties and fans}

Let $T^n \subset \Cstarn$ be the real torus with Lie algebra $\t$ and denote the dual pairing between $\t$ and the dual space $\mathfrak{t}^{*}$ by $\langle \cdot\,,\cdot\rangle$. There is a natural integer lattice $\Gamma \simeq \Z^n \subset \t$ comprising all $\lambda \in \t$ such that $\operatorname{exp}(\lambda) \in T^n$ is the identity. This then induces a dual lattice $\Gamma^* \subset \t^*$. We have the following combinatorial definition.

\begin{definition}
 A \emph{fan} $\Sigma$ in $\t$ is a finite set of rational polyhedral cones $\sigma$ satisfying:

	\begin{enumerate}
		\item For every $\sigma \in \Sigma$, each face of $\sigma$ also lies in $\Sigma$.
		\item For every pair $\sigma_1, \sigma_2 \in \Sigma$, $\sigma_1 \cap \sigma_2$ is a face of each.
	\end{enumerate}
\end{definition}

To each fan $\Sigma$ in $\t$, one can associate a toric variety $X_\Sigma$. Heuristically, $\Sigma$ contains all the data necessary to produce a partial equivariant compactification of $\Cstarn$, resulting in $X_\Sigma$. More concretely, one obtains $X_\Sigma$ from $\Sigma$ as follows. For each $n$-dimensional cone $\sigma \in \Sigma$, one constructs an affine toric variety $U_\sigma$ which we first explain. We have the dual cone $\sigma^{\vee}$ of $\sigma$. Denote by $S_\sigma$ the semigroup of those lattice points which lie in $\sigma^{\vee}$ under addition. Then one defines the semigroup ring, as a set, as all finite sums of the form
	\begin{equation*}
		\C[S_\sigma] = \left\{ \left.\sum \lambda_s s \, \right| \, s \in S_\sigma \right\},
	\end{equation*}
with the ring structure defined on monomials by $\lambda_{s_1}s_1\cdot \lambda_{s_2}s_2  = (\lambda_{s_1}\lambda_{s_2})(s_1+ s_2)$ and extended in the natural way. The affine variety $U_\sigma$ is then defined to be $\text{Spec}(\C[S_\sigma])$. This automatically comes endowed with a $\Cstarn$-action with a dense open orbit. This construction can also be applied to the lower dimensional cones $\tau \in \Sigma$. If $\sigma_1 \cap \sigma_2 = \tau$, then there is a natural way to map $U_\tau$ into $U_{\sigma_1}$ and $U_{\sigma_2}$ isomorphically. One constructs $X_\Sigma$ by declaring the collection of all $U_\sigma$ to be an open affine cover of $X_{\Sigma}$ with transition functions determined by $U_\tau$. This identification is also reversible.

\begin{prop}[{\cite[Corollary 3.1.8]{cox}}]\label{fann}
Let $M$ be a smooth toric manifold. Then there exists a fan $\Sigma$ such that $M \simeq X_\Sigma$.
\end{prop}

 \begin{prop}[{\cite[Theorem 3.2.6]{cox}, Orbit-Cone Correspondence}]\label{orbitcone} The $k$-dimensional cones $\sigma \in \Sigma$ are in a natural one-to-one correspondence with the $(n-k)$-dimensional orbits $O_\sigma$ of the $\Cstarn$-action on $X_\Sigma$.
 \end{prop}

In particular, each ray $\sigma \in \Sigma$ determines a unique torus-invariant divisor $D_\sigma$. As a consequence, a torus-invariant Weil divisor $D$ on $X_\Sigma$ naturally determines a polyhedron $P_D \subset \mathfrak{t}^{*}$. Indeed, we can decompose $D$ uniquely as $D = \sum_{i=1}^N a_i D_{\sigma_i}$, where $\{\sigma_i\}_{i}\subset\Sigma$ is the collection of rays. Then by assumption, there exists a unique minimal lattice element $\nu_i \in \sigma_i \cap \Gamma$. $P_{D}$ is then given by

 \begin{equation} \label{eqnB2}
 	P_D = \left\{ x \in \mathfrak{t}^{*} \: | \: \langle \nu_i, x \rangle \geq - a_i \right\} = \bigcap_{i = 1}^N K(\nu_i, -a_i).
 \end{equation}

\subsubsection{K\"ahler metrics on toric varieties}\label{finito}

For a given toric manifold $M$ endowed with a Riemannian metric $g$ invariant under the action of the real torus $T^n \subset \Cstarn$ and K\"ahler with respect to the underlying complex structure
of $M$, the K\"ahler form $\omega$ of $g$ is also invariant under the $T^n$-action. We call such a manifold a \emph{toric K\"ahler manifold}.
In what follows, we always work with a fixed complex structure on $M$.

Hamiltonian K\"ahler metrics have a useful characterisation due to Guillemin.

\begin{prop}[{\cite[Theorem 4.1]{gilly}}]\label{propB6}
	Let $\omega$ be any $T^n$-invariant K\"ahler form on $M$. Then the $T^{n}$-action is Hamiltonian with respect to $\omega$ if and only if the restriction of $\omega$ to the dense orbit $\Cstarn \subset M$ is exact, i.e., there exists a $T^{n}$-invariant potential $\phi$ such that
	\begin{equation*}
		\omega = 2i\p\bp \phi.
	\end{equation*}
\end{prop}

Fix once and for all a $\Z$-basis $(X_1,\ldots,X_n)$ of $\Gamma \subset \t$. This in particular induces a background coordinate system $\xi=(\xi^1, \dots, \xi^n)$ on $\t$.
Using the natural inner product on $\t$ to identify $\t \cong \t^*$, we can also identify $\t^* \cong \R^n$.
For clarity, we will denote the induced coordinates on $\t^*$ by $x=(x^1,\ldots, x^n)$. Let $(z_1, \dots, z_n)$ be the natural
coordinates on $\Cstarn$ as an open subset of $\C^n$. There is a natural diffeomorphism $\text{Log}:
\Cstarn \to \t \times T^n$ which provides a one-to-one correspondence between $T^n$-invariant smooth functions on
 $\Cstarn$ and smooth functions on $\t$. Explicitly,
\begin{equation}\label{diffeoo}
(z_1, \dots, z_n)\xmapsto{\operatorname{Log}}(\log(r_1), \dots, \log(r_n), \theta_1, \dots, \theta_n)=(\xi_{1},\ldots,\xi_{n},\,\theta_{1},\ldots,\theta_{n}),
\end{equation}
where $z_j = r_j e^{i \theta_j}$,\,$r_{j}>0$. Given a function $H(\xi)$ on $\t$, we can extend $H$ trivially to $\t \times T^n$ and pull back by Log to
obtain a $T^n$-invariant function on $\Cstarn$. Clearly, any $T^n$-invariant function on $\Cstarn$ can be written in this form.

Choose any branch of $\log$ and write $w = \log(z)$. Then clearly $w = \xi + i \theta$, where $\xi=(\xi^1,\ldots,\xi^n)$ are real coordinates on $\t$
(or, more precisely, there is a corresponding lift of $\theta$ to the universal cover with respect to which this equality holds),
and so if $\phi$ is $T^n$-invariant and $\omega = 2i \p \bp \phi$, then we have that
\begin{equation}\label{e:T5}
	\omega = 2i\frac{\p^2 \phi}{ \p w^i \p\bar{w}^j} dw_i \wedge d\bar{w}_j = \frac{\p^2 \phi}{ \p \xi^i \p\xi^j} d\xi^i \wedge d\theta^j.
\end{equation}
In this setting, the metric $g$ corresponding to $\omega$ is given on $\t \times T^n$ by
\begin{equation}\label{metricc}
	g = \phi_{ij}(\xi)d\xi^i d\xi^j + \phi_{ij}(\xi)d\theta^i d\theta^j,
\end{equation}
and the moment map $\mu$ as a map $\mu: \t \times T^n \to \t^*$ is defined by the relation
		\begin{equation*}
		\langle \mu(\xi, \theta), b \rangle = \langle \nabla \phi(\xi), b \rangle
	\end{equation*}
for all $b \in \t$, where $\nabla \phi$ is the Euclidean gradient of $\phi$.
The $T^n$-invariance of $\phi$ implies that it depends only on $\xi$ when considered a function on $\t \times T^n$
via \eqref{diffeoo}. Since $\omega$ is K\"ahler, we see from \eqref{e:T5} that the Hessian of $\phi$ is positive definite so that $\phi$ itself is strictly convex.
In particular, $\nabla \phi$ is a diffeomorphism onto its image.
Using the identifications mentioned above, we view $\nabla \phi$ as a map from $\t$ into an open subset of $\t^*$.

\subsubsection{K\"ahler-Ricci solitons on toric manifolds}

Next we define what we mean by a shrinking K\"ahler-Ricci soliton in the toric category.
\begin{definition}
A complex $n$-dimensional shrinking gradient K\"ahler-Ricci soliton $(M,\,g,\,X)$ with complex structure $J$ and K\"ahler form $\omega$ is \emph{toric} if $(M,\,\omega)$ is a toric K\"ahler
manifold as in Definition \ref{toricmanifold} and $JX$ lies in the Lie algebra $\t$ of the underlying real torus $T^{n}$ that acts on $M$. In particular, the zero set of $X$ is compact.
\end{definition}

It follows from \cite{wyliee} that $\pi_{1}(M)=0$, hence the induced real $T^n$-action is automatically Hamiltonian with respect to $\omega$.
Working on the dense orbit $\Cstarn \subset M$, the condition that a vector field $JY$ lies in $\t$ is equivalent to saying that in the coordinate system $(\xi^1,\ldots,\xi^n,\,\theta_{1},\ldots,\theta_{n})$
from \eqref{diffeoo}, there is a constant $b_Y=(b_{Y}^{1},\ldots,b_{Y}^{n})\in \R^n$ such that
\begin{equation}\label{eqnY4}
	JY =  b_Y^i \frac{\p}{\p\theta^i}\qquad\textrm{or equivalently,}\qquad Y =   b_Y^i \frac{\p}{\p\xi^i}.
\end{equation}
From Proposition \ref{propB6}, we know that $\mathcal{L}_{X}\omega=2i\p\bp X(\phi)$. In addition,
the function $X(\phi)$ on $\Cstarn$ can be written as $\langle b_X, \nabla \phi \rangle = b_X^j \frac{\p\phi}{\p\xi^j}$,
where $b_{X}\in\mathbb{R}^{n}$ corresponds to the soliton vector field $X$ via \eqref{eqnY4}.
These observations allow us to write the shrinking soliton equation \eqref{krseqn} as a real Monge-Amp\`ere equation for $\phi$ on $\R^n$.

\begin{prop}[{\cite[Proposition 2.6]{charlie}}]
Let $(M,\,g,\,X)$ be a toric shrinking gradient K\"ahler-Ricci soliton with K\"ahler form $\omega$. Then
there exists a unique smooth convex real-valued function $\phi$ defined on the dense orbit $\Cstarn\subset M$ such that $\omega=2i\partial\bar{\partial}\phi$
and
\begin{equation} \label{realMA}
	\det(\phi_{ij})=e^{-2\phi+\langle b_{X},\,\nabla\phi\rangle}.
\end{equation}
\end{prop}

A priori, the function $\phi$ is defined only up to addition of a linear function.
However, \eqref{realMA} provides a normalisation for $\phi$ which in turn provides a normalisation for $\nabla\phi$, the moment map of the action.
The next lemma shows that this normalisation coincides with that for the moment map as defined in \cite[Definition 5.16]{cds}.

\begin{lemma}\label{normal}
Let $(M,\,g,\,X)$ be a toric complete shrinking gradient K\"ahler-Ricci soliton with complex structure $J$ and K\"ahler form $\omega$
with soliton vector field $X=\nabla^{g}f$ for a smooth real-valued function $f:M\to\mathbb{R}$.
Let $\phi$ be given by Proposition \ref{propB6} and normalised by \eqref{realMA}, let $JY\in\mathfrak{t}$, and let $u_{Y}=\langle\nabla\phi,\,b_{Y}\rangle$ be the Hamiltonian potential of $JY$
with $b_{Y}$ as in \eqref{eqnY4} so that $\nabla^{g}u_{Y}=Y$. Then $\mathcal{L}_{JX}u_{Y}=0$ and $\Delta_{\omega}u_{Y}+u_{Y}-\frac{1}{2}Y\cdot f=0$.
\end{lemma}

To see the equivalence with \cite[Definition 5.16]{cds},
simply replace $Y$ with $JY$ in this latter definition as here we assume that $JY\in\mathfrak{t}$, contrary to the convention in \cite[Definition 5.16]{cds} where it is assumed that
$Y\in\mathfrak{t}$.

\begin{proof}[Proof of Lemma \ref{normal}]
By definition, we have that
$$d\left(\mathcal{L}_{JX}u_{Y}\right)=\mathcal{L}_{JX}(du_{Y})=-\mathcal{L}_{JX}(\omega\lrcorner JY)=0,$$
where we have used the fact that $\mathcal{L}_{JX}\omega=0$ and $[JX,\,JY]=0$. $\mathcal{L}_{JX}u_{Y}$ is therefore equal to a constant
which must be zero as $JX$ has a zero because $X=\nabla^{g}f$ and $f$ is proper and bounded from below (cf.~Theorem \ref{theo-basic-prop-shrink}), hence attains a local minimum.
This proves the $JX$-invariance of $u_{Y}$.

The final equation follows by differentiating \eqref{realMA} with respect to $Y$. Indeed,
from \eqref{e:T5} and \eqref{metricc} we see that on the dense orbit,
\begin{equation*}
\begin{split}
\frac{\omega^{n}}{n!}&=\operatorname{vol}_{g}=\det(\phi_{ij})\,d\xi^{1}\wedge d\theta^{1}\wedge\ldots\wedge d\xi^{n}\wedge d\theta^{n}\\
&=\frac{\det(\phi_{ij})}{(-2i)^{n}}\,dw^{1}\wedge d\bar{w}^{1}\wedge\ldots\wedge dw^{n}\wedge d\bar{w}^{n}.
\end{split}
\end{equation*}
Recalling that $f$ denotes the Hamiltonian potential of $JX\in\mathfrak{t}$ so that $f=\langle b_{X},\,\nabla\phi\rangle$ on the dense orbit, \eqref{realMA}
may therefore be rewritten as
$$\log\det\left(\frac{(-2i)^{n}\omega^{n}}{n!dw^{1}\wedge d\bar{w}^{1}\wedge\ldots\wedge dw^{n}\wedge d\bar{w}^{n}}\right)+2\phi-f=0.$$
Differentiating along $Y$, this yields the relation
\begin{equation*}
\begin{split}
0&=Y\cdot\log\det\left(\frac{(-2i)^{n}\omega^{n}}{n!dw^{1}\wedge d\bar{w}^{1}\wedge\ldots\wedge dw^{n}\wedge d\bar{w}^{n}}\right)+2Y\cdot\phi-Y\cdot f\\
&=\operatorname{tr}_{\omega}\mathcal{L}_{Y}\omega+2u_{Y}-Y\cdot f\\
&=2\Delta_{\omega}u_{Y}+2u_{Y}-Y\cdot f,\\
\end{split}
\end{equation*}
where we have made use of \cite[Lemma 2.5]{charlie} in the last line. From this, the result follows.
\end{proof}

Given the normalisation \eqref{realMA}, the next lemma identifies the image of the moment map $\mu=\nabla\phi$.

\begin{lemma}[{\cite[Lemmas 4.4 and 4.5]{charlie}}]\label{note}
Let $(M,\,g,\,X)$ be a complete toric shrinking gradient K\"ahler-Ricci soliton, let $\{D_i\}$ be the prime $\Cstarn$-invariant divisors in $M$, and let $\Sigma \subset \t$ be the fan determined by
Proposition \ref{fann}. Let $\sigma_i\in\Sigma$ be the ray corresponding to $D_i$ with minimal generator $\nu_i \in \Gamma$.
\begin{enumerate}
  \item There is a distinguished Weil divisor representing the anticanonical class $-K_{M}$ given by
  \begin{equation*}
		-K_M = \sum_i D_i
	\end{equation*}
	whose associated polyhedron (cf.~\eqref{eqnB2}) is given by
 	\begin{equation}\label{xmas}
		P_{-K_M} = \left\{ x \: | \: \langle \nu_i, x \rangle \geq -1 \right\}
	\end{equation}
which is strongly convex and has full dimension in $\mathfrak{t}^{*}$. In particular, the origin lies in the interior of $P_{-K_{M}}$.
  \item If $\mu$ is the moment map for the induced real $T^n$-action normalised by \eqref{realMA}, then the image of $\mu$ is precisely $P_{-K_{M}}$.
\end{enumerate}
	\end{lemma}

\subsubsection{The weighted volume functional}\label{weighted}

As a result of Lemma \ref{normal}, we can now define the weighted volume functional.

\begin{definition}[{Weighted volume functional, \cite[Definition 5.16]{cds}}]\label{weightedvol}
Let $(M,\,g,\,X)$ be a complex $n$-dimensional toric shrinking gradient K\"ahler-Ricci soliton with K\"ahler form $\omega=2i\partial\bar{\partial}\phi$
 on the dense orbit with $\phi$ strictly convex with moment map $\mu=\nabla\phi$ normalised by \eqref{realMA}. Assume that the fixed point set of the torus is compact and
 recall that $$\Lambda_{\omega}:=\{Y\in\mathfrak{t}\:|\:\textrm{$\langle\mu,\,Y\rangle$ is proper and bounded below}\}\subseteq\mathfrak{t}.$$ Then the \emph{weighted volume functional} $F:\Lambda_{\omega}\to\mathbb{R}$ is defined by
	\begin{equation*}
		F_{\omega}(v) = \int_M e^{-\langle \mu,\,v \rangle} \omega^n.
	\end{equation*}	
\end{definition}

As the fixed point set of the torus is compact by definition, $F_{\omega}$ is well-defined by the non-compact version of the Duistermaat-Heckman formula \cite{wu}
(see also \cite[Theorem A.3]{cds}). This leads to two important lemmas concerning the weighted volume functional in the toric category, the independence of $\Lambda_{\omega}$ and $F_{\omega}$ from the choice of shrinking soliton $\omega$.

\begin{lemma}\label{one}
$\Lambda_{\omega}$ is independent of the choice of toric shrinking K\"ahler-Ricci soliton $\omega$ in Definition \ref{weightedvol}
and is given by $\Lambda_{\omega}=\operatorname{int}(C^{\vee})$, where $C:=\{x\:|\:\langle\nu_{i},\,x\rangle\geq0\}$ and $\{\nu_{i}\}$ are as in
Lemma \ref{note}.
\end{lemma}

\begin{proof}
Recall from  Proposition \ref{identify} that $\Lambda_{\omega}$ is given by $\operatorname{int}(\mathcal{C}(\mu_{\omega}(M))^{\vee})$, where
the moment map $\mu_{\omega}$ with respect to $\omega$, normalised by \eqref{realMA}, depends on $\omega$. However, no matter the choice of $\omega$ in Definition \ref{weightedvol}, the normalisation \eqref{realMA} implies by Lemma \ref{note}(ii) that the image of $M$ under the moment map is always given by $P_{-K_{M}}$, a fixed polytope determined solely by the torus action. $\Lambda_{\omega}$ is therefore independent of the choice of $\omega$ in Definition \ref{weightedvol}. Finally, the asymptotic cone of this polytope (as a subset of $\mathfrak{t}^{*}$) is, by Lemma \ref{alpha}, given by $C$. This leads to the desired expression for $\Lambda_{\omega}$.
\end{proof}

Note that $C$ is always a strongly convex rational polyhedral cone in $\t$, although not necessarily of full dimension, whereas $C^\vee$ is always full dimensional, although not necessarily strongly convex.

\begin{lemma}\label{two}
$F_{\omega}$ is independent of the choice of toric shrinking K\"ahler-Ricci soliton $\omega$ in Definition \ref{weightedvol}. Moreover,
after identifying $\Lambda_{\omega}$ with a subset of $\mathbb{R}^{n}$ via \eqref{eqnY4}, $F_{\omega}$
is given by $F_{\omega}(v)=(2\pi)^n \int_{P_{-K_M}} e^{-\langle v,\,x \rangle }\,dx$, where $x=(x^1,\ldots,x^n)$ denotes coordinates on $\mathfrak{t}^{*}$ dual to
the coordinates $(\xi^{1},\ldots,\xi^{n})$ on $\t$ introduced in Section \ref{finito}.
\end{lemma}

\begin{proof}
We first show that the given integral is finite. To demonstrate this, it suffices to show that $\langle v,\,x\rangle>0$
on the complement of a compact subset of $P_{-K_{M}}$. To this end, recall that $0\in\operatorname{int}(P_{-K_{M}})$ so that the intersection of
the hyperplane $\{x\in\mathbb{R}^{n}\:|\:\langle v,\,x\rangle=0\}$ with $P_{-K_{M}}$ is non-empty. We claim that the polyhedron
$Q:=\{x\in P_{-K_{M}}\:|\:\langle v,\,x\rangle\leq0\}$ is compact. Indeed, by Lemma \ref{beta}, $Q$ is compact if and only if $\mathcal{C}(Q)=\{0\}$.
To derive a contradiction, assume that there exists a non-zero vector $w\in\mathcal{C}(Q)$.
Then from the definition of the asymptotic cone, one can see that $Q$ contains a ray of the form
$x_{0}+tw,\,t\geq0,$ for some $x_{0}\in Q$. Taking the inner product with $v$, it follows that
$\langle v,\,x_{0}+tw\rangle>0$ for $t\gg0$ because $\langle v,\,w\rangle>0$
by virtue of the fact that
$$Q\subseteq P\implies \mathcal{C}(Q)\subseteq\mathcal{C}(P)\implies
\mathcal{C}(P)^{\vee}\subseteq\mathcal{C}(Q)^{\vee}
\implies\operatorname{int}(\mathcal{C}(P)^{\vee})\subseteq\operatorname{int}(\mathcal{C}(Q)^{\vee}).$$
This yields the desired contradiction.

Now, no matter the choice of shrinking soliton, the map $\nabla \phi: \t \to P_{-K_M}$ defines a diffeomorphism with image the fixed polytope $P_{-K_M}$
thanks to the normalisation given by \eqref{realMA}. The independence of $F_{\omega}$ from $\omega$ and the given expression then follows from the following computation, where
$\nabla \phi: \t \to P_{-K_M}$ is used as a change of coordinates:
\begin{equation*}
\begin{split}
	F_{\omega}(v) &= \int_M e^{-\langle\mu,\,v\rangle}\omega^n = \int_{\t \times T^{n}}e^{-\langle\nabla\phi(\xi),\,v \rangle}\det(\phi_{ij}(\xi))\,d\xi d\theta \\
		&= (2\pi)^n \int_{\t}e^{-\langle \nabla \phi (\xi), v \rangle}\det(\phi_{ij}(\xi))\,d\xi = (2\pi)^n \int_{P_{-K_M}} e^{-\langle x,v \rangle }\,dx.
\end{split}
\end{equation*}
\end{proof}

Thus, we henceforth drop the subscript $\omega$ from $F_{\omega}$ and $\Lambda_{\omega}$ when working in the toric category. The
functional $F:\Lambda\to\mathbb{R}$ is proper in this category, hence attains a critical point in $\Lambda$.

\begin{prop}[{\cite[Proof of Proposition 3.1]{charlie}}]\label{properr}
The functional $F(v)=(2\pi)^n \int_{P_{-K_M}} e^{-\langle v,\,x \rangle }\,dx$ is proper on $\Lambda$.
\end{prop}

In general, such a critical point turns out to be unique and characterises the soliton vector field of a complete shrinking gradient K\"ahler-Ricci soliton.

\begin{theorem}[{\cite[Lemma 5.17]{cds}, \cite[Theorem 1.1]{caoo}}]\label{thmB13}
Let $(M,\,g,\,X)$ be a complete shrinking gradient K\"ahler-Ricci soliton with complex structure $J$, K\"ahler form $\omega$, and bounded Ricci curvature.
Then $JX\in\Lambda_{\omega}$, $F_{\omega}$ is strictly convex on $\Lambda_{\omega}$, and $JX$ is the unique critical point of $F_{\omega}$ in $\Lambda_{\omega}$.
\end{theorem}

Having established in Lemmas \ref{one} and \ref{two} that in the toric category the weighted volume functional $F$ and its domain $\Lambda$ are determined solely by the polytope $P_{-K_{M}}$ which itself, by Lemma \ref{note}, depends only on the torus action on $M$ (i.e., is independent of the choice of shrinking soliton), and having an explicit expression for $F$ given by Lemma \ref{two},
after using the torus action to identify $P_{-K_{M}}$ via \eqref{xmas}, we can determine explicitly the soliton vector field of a hypothetical toric shrinking gradient K\"ahler-Ricci soliton on $M$.
We illustrate how to do this in the following examples.

\begin{example}
Consider $\P^1$ with the $\Cstar$-action given by $\lambda \cdot [Z_0:Z_1] = [\lambda Z_0: Z_1]$. Then its torus-invariant divisors are $D_0 = [0:1]$ and $D_\infty = [1:0]$. The corresponding fan in $\R$ is given by $\Sigma_{\P^1} = \{0, [0, \infty), (-\infty, 0]\}$ and $-K_{\P^1} = D_0 + D_1$, the associated polyhedron $P_{-K_{\P^{1}}}$ of which can naturally be identified with the interval $[-1,1]\subset\R$. The Fubini-Study metric $\omega_{\mathbb{P}^{1}}$ is K\"ahler-Einstein and in particular, $2\omega_{\mathbb{P}^{1}}$
is a shrinking gradient K\"ahler-Ricci soliton on $\mathbb{P}^{1}$ with soliton vector field $X=0$. Working with $2\omega_{\mathbb{P}^{1}} \in 2\pi c_1(-K_{\P^1})$, on the dense orbit $\Cstar \subset \P^1$, $2\omega_{\mathbb{P}^{1}}$ has K\"ahler potential
	\begin{equation*}
		\phi_{2\omega_{\mathbb{P}^{1}}}:= \log\left(1 + |z|^2\right) - \frac{1}{2}\log\left(4|z|^2\right) = \log\left(e^{2\xi} + 1\right) - \xi -\log(2),
	\end{equation*}
so that $\omega_{\mathbb{P}^1} = 2i\p\bp\phi_{2\omega_{\mathbb{P}^{1}}}$. It is then straightforward to verify that $\phi_{2\omega_{\mathbb{P}^{1}}}$ satisfies \eqref{realMA} with $b_X = 0$ and that the image of $\frac{\p\phi_{2\omega_{\mathbb{P}^{1}}}}{\p\xi}$ is the interval $[-1,1]$. The weighted volume functional is then given by
		\begin{equation*}
		F_{\P^1}(v) =2\pi\int_{-1}^1 e^{-vx}\,dx.
	\end{equation*}
This is defined for all $v \in \R$ and indeed, the asymptotic cone of the compact polytope $[-1,1]$ is just the point $0$ so that $C^{\vee} = \R$. Clearly $F'(v) = 0$ if and only if $v = 0$,
as expected.
\end{example}

\begin{example}
Consider $\C$ endowed with the standard $\Cstar$-action. Then there is only one torus-invariant divisor, namely $D = \{0\}$. The fan in $\R$ is simply $\Sigma_{\C} = \{0, [0, \infty)\}$ and $-K_{\C} = D$ with the corresponding polyhedron given by $P_{-K_{\C}} = [-1,\infty)$. On the dense orbit $\Cstar \subset \C$, the Euclidean metric $\omega_{\mathbb{C}}$ has K\"ahler potential
		 \begin{equation*}
	 	\phi_{\omega_{\mathbb{C}}} = \frac{1}{4}|z|^2 - \log(|z|^2) - 1 = \frac{1}{4}e^{2\xi} - \xi -1.
	 \end{equation*}
This satisfies \eqref{realMA} with $b_X = 1$ and the image of $\frac{\p \phi_{\omega_{\mathbb{C}}}}{\p\xi}$ is $[-1,\infty)$. The asymptotic cone $C$ of $[-1, \infty)$ is given by $[0, \infty)$ and
accordingly, the weighted volume functional
		 \begin{equation*}
		F_{\C}(v) = 2\pi\int_{-1}^\infty e^{-vx}\,dx
	\end{equation*}
is only defined on the interior of the dual cone $C^{\vee}$, namely $(0, \infty)$. We compute:
	\begin{equation*}
		F'(v) = -2\pi\int_{-1}^\infty x e^{-vx}\,dx = \frac{e^{v}}{v^2}(1-v).
	\end{equation*}
Hence, as expected, $F_{\C}$ has a unique critical point at $v = 1$, with the corresponding soliton vector field on $\mathbb{C}$ given by
		\begin{equation*}
		X = \frac{\p}{\p\xi} = r\frac{\p}{\p r},
	\end{equation*}
where $z = re^{i\theta} \in \C$ for $r>0$ and $\xi = \log(r)$.
\end{example}

\begin{example}\label{vf-id2}
Next we consider the Cartesian product $\C\times\P^1$ of the previous two examples. We equip $\C\times\P^1$ with the product $\Cstarsqr$-action and denote by $\t_1$ and $\t_2$ the Lie algebras of the real $S^1$'s that act on $\C$ and $\P^1$, respectively. Then we have an obvious solution to \eqref{realMA} given by the product metric $\omega_{\mathbb{C}}+2\omega_{\mathbb{P}^{1}}$ together with the soliton vector field $X_{\C\times \P^1} = X_{\C} + X_{\P^{1}}=r \frac{\p}{\p r}$ with $r=|z|$, $z$ the complex coordinate on the $\C$-factor.
Explicitly, the fan $\Sigma_{\C \times \P^1}$ comprises products $\sigma_1 \times \sigma_2 \subset \t_1 \oplus \t_2$, where $\sigma_1 \in \Sigma_{\C}$ and $\sigma_2\in \Sigma_{\P^1}$. The polyhedron $P_{-K_{\C \times \P^1}} \subset \t = \t_1 \oplus \t_2$ can be identified with the subset of $\R^2$ defined by the inequalities	
\begin{equation*}
P_{-K_{\C \times \P^1}} = \left\{ (x_1, x_2) \in \R^2 \: | \: x_1 \geq -1, \: -1 \leq x_2 \leq 1 \right\}.
\end{equation*}
From this, one can easily see that if $v = v_1 + v_2$ with $v_1 \in \t_1$ and $v_2 \in \t_2$, then
		\begin{equation*}
		F_{\C \times \P^1}(v) = F_{\C}(v_1) F_{\P^1}(v_2).
	\end{equation*}
The fact that $F_{\C}$ and $F_{\P^1}$ are convex and positive implies that $F'_{\C \times \P^1}(v) = 0$ if and only if $F'_{\C}(v_1) = F'_{\P^1}(v_2) = 0$, as expected.
\end{example}

\begin{example}\label{vf-id}
Let $M=\operatorname{Bl}_{p}(\C \times \P^1)$ denote the blowup of a fixed point $p$ of the $(\mathbb{C}^{*})^{2}$-action on $\C\times\mathbb{P}^{1}$ and write $J$ for the complex structure on $M$.
Then $M$ inherits a natural $\Cstarsqr$-action with respect to which the blowdown map $\pi:M\to\C\times\P^1$ is $(\mathbb{C}^{*})^{2}$-equivariant. In terms of the toric data, the exceptional divisor $E$ of $\pi$ defines an additional invariant divisor and the polyhedron $P_{\C \times \P^1}$ is modified accordingly:
	\begin{equation*}
			P_{-K_{M}} = \left\{ (x_1, x_2) \in \R^2 \: | \: x_1 \geq -1, \: -1 \leq x_2 \leq 1, \: x_1 + x_2 \geq -1 \right\}.
	\end{equation*}
Here, the new face with inner normal $\nu_E = (1,\,1)$ corresponds to $E$. Define two auxiliary functions $F_1$ and $F_2$ of a real variable $t>0$ by
	\begin{equation*}
		F_1(t) =\int_{P_{-K_M}} x_1 e^{-t(2x_1 + x_2)}\,dx_1dx_2,\qquad F_2(t) =\int_{P_{-K_M}} x_2 e^{-t(2x_1 + x_2)}\,dx_1dx_2.
	\end{equation*}
These functions are, up to a scaling factor of $-(2\pi)^{-2}$, the components of the gradient of
the weighted volume functional $F_{M}:\Lambda\to\R$ of $M$ restricted to the ray generated by $(2,\,1) \in \t$.
Thus, if there exists some $\lambda>0$ such that $F_1(\lambda)=F_2(\lambda) = 0$, then the point $(2\lambda, \lambda) \in \R^2$ would be a critical point of $F_{M}$. First, we claim that $F_2(t) \equiv 0$. Indeed, computing directly, we see that
\begin{equation*}
\begin{split}
		F_2(t) &= \int_0^1\int_{-1}^{\infty}  x_2 e^{-t(2x_1 + x_2)}\,dx_1dx_2 + \int_{-1}^{0}\int_{-(x_2+1)}^{\infty } x_2 e^{-t(2x_1 + x_2)}\,dx_1dx_2 \\
				&= \frac{t^{-3}}{2}e^t(e^t -1) - \frac{t^{-2}}{2}e^t - \frac{t^{-3}}{2}e^t(e^t -1) + \frac{t^{-2}}{2}e^t = 0.
\end{split}
\end{equation*}
Since $F_{M}(t(2,1))$ is proper and convex as a function of $t>0$ (cf.~Proposition \ref{properr} and Theorem \ref{thmB13}), this implies
that there is a $\lambda>0$ such that both $F_1(\lambda)$ and $F_2(\lambda)$ vanish simultaneously. We next determine the value of $\lambda$. Since
\begin{equation*}
\begin{split}		
F_1(t) &= - \int_{-1}^{0}\int_{-(x_1+1)}^{1}x_1e^{-t(2x_1 + x_2)} dx_2 dx_1 - \int_{0}^{\infty}\int_{-1}^{1}x_1e^{-t(2x_1 + x_2)} dx_2 dx_1 \\
				&= \frac{t^{-2}}{2}e^t + \frac{t^{-3}}{2}\sinh(t) - e^t\left(t^{-2}e^t + t^{-3}(e^t -1) \right)- \frac{t^{-3}}{2}\sinh(t)  \\
				&= \frac{t^{-3}}{2}e^t \left(2e^t(1-t) -(2-t)\right),
\end{split}
\end{equation*}
we see that $F_1(\lambda) = 0$ for $\lambda$ the $x$-coordinate of the unique non-zero point of intersection of the graphs of $G_1(t) = 2e^t(1-t)$ and $G_2(t) = 2-t$.
In particular, $\lambda$ cannot be equal to $1,\,2,\,\frac{1}{2},$ or indeed any algebraic number. Numerical approximations in fact give
$\lambda \approx0.64$.

Let $(z_1,\,z_2)$ be complex coordinates on the dense orbit $\Cstarsqr \subset \Cstar \times \P^1 \subset M$. Writing $z_j = r_j e^{i\theta_j}$ with $r_{j}>0$, set $\xi_j = \log(r_j)$ as before.
Then the soliton vector field $X$ on $M$ may be written as
		\begin{equation}\label{vfield}
		X = \lambda\left(2 \frac{\p}{\p\xi^1} + \frac{\p}{\p\xi^2} \right) = \lambda\left( 2r_1\frac{\p}{\p r_1} + r_2 \frac{\p}{\p r_2}\right).
	\end{equation}
\end{example}

\section{Proof of Theorem \ref{mainthm1}}\label{proof1}

Consider a complete non-compact shrinking gradient K\"ahler-Ricci soliton $(M,\,g,\,X)$
with bounded scalar curvature with complex structure $J$ and with soliton vector field
$X=\nabla^{g}f$ for a smooth real-valued function $f:M\to\mathbb{R}$. Then $f$ is proper and bounded from below (cf.~Theorem \ref{theo-basic-prop-shrink}), hence attains a minimum,
and $g$ complete implies that $X$ is complete \cite{Zhang-Com-Ricci-Sol}. Let $G^{X}_{0}$ denote the connected component of the identity of the holomorphic isometries of $(M,\,J,\,g)$ that
commute with the flow of $X$. Since $g$ has bounded Ricci curvature, $G^{X}_{0}$ is a compact Lie group
by \cite[Lemma 5.12]{cds} and $X$ being complete implies that $JX$ is complete by \cite[Lemma 2.35]{cds}. Moreover, $JX$ is Killing by \cite[Lemma 2.3.8]{fut2}. Hence the closure of the flow of $JX$ in $G^{X}_{0}$ yields the holomorphic isometric action of a real torus $T$ on $(M,\,J,\,g)$ with Lie algebra $\mathfrak{t}$ containing $JX$. Compactness of the zero set of $X$ \cite[Lemma 2.26]{cds} and hence $JX$ implies that the fixed point set of $T$ is compact. Moreover, as $f$ attains a minimum, $T$ will have at least one fixed point.
Finally, as $H^{1}(M)=0$ by \cite{wyliee}, $T$ will act on $M$ in a Hamiltonian fashion.

Using results from $J$-holomorphic curves (cf.~Section \ref{Jholo}), we identify the candidate non-compact complex surfaces
that may admit a complete shrinking gradient K\"ahler-Ricci soliton with bounded scalar curvature whose soliton vector field has an integral curve along which the scalar curvature does not tend to zero. We first work under the assumption of simple connectedness and classify up to diffeomorphism.

\begin{prop}[Smooth classification]\label{smoothclass}
Let $(M,\,g,\,X)$ be a two-dimensional simply connected complete non-compact shrinking gradient K\"ahler-Ricci soliton with
bounded scalar curvature $\RR_{g}$ with $X=\nabla^{g}f$ for some smooth function
$f:M\rightarrow\mathbb{R}$. Assume that $X$ has an integral curve along which $\RR_{g}\not\to0$.
Then $M$ is diffeomorphic to either $\mathbb{C}\times\mathbb{P}^{1}$ or to $\operatorname{Bl}_{p}(\mathbb{C}\times\mathbb{P}^{1})$, that is, the blowup of $\mathbb{C}\times\mathbb{P}^{1}$ at one point $p$. In the former case, the zero set of $X$ is contained in a unique $\mathbb{P}^{1}$ and in the latter case,
in the pre-image under the blowup map of the $\mathbb{P}^{1}$-fibre containing the blowup point.
\end{prop}

\begin{proof}
Since there exists a point of $M$ where $\RR_{g}\neq0$, $g$ is non-flat and so globally we know that $\RR_{g}>0$ \cite{pigola1}. As $X$ has an integral curve along which $\RR_{g}\not\to0$
by assumption, this means that there exists $\varepsilon>0$ and a sequence of points $\{x_{i}\}_{i}$
lying along this integral curve going off to infinity as $i\to\infty$ such that $\RR_{g}(x_{i})>\varepsilon$.
By assumption $\RR_{g}$ is bounded, thus we read from \cite[Theorem 1.3]{wang22} that the norm of the full curvature tensor $\operatorname{Rm}(g)$ of $g$ is bounded.
It subsequently follows from \cite[Corollary 4.1]{naber} and the classification of real three-dimensional complete shrinking gradient Ricci solitons \cite[Theorem 1.2]{wang22} that the sequence of pointed manifolds $(M,\,g,\,x_{i})$, after passing to a subsequence if necessary, converges in the smooth pointed Cheeger-Gromov sense to
$(\widehat{M},\,\hat{g},\,\hat{p})$, where $\hat{p}\in\widehat{M}$ is a base point and
$(\widehat{M},\,\hat{g})$ is isometric to $\mathbb{R}^{4}$ endowed with the flat metric, or
$\widehat{M}$ is diffeomorphic to $\mathbb{R}^{2}\times S^{2}$ or to the $\mathbb{Z}_{2}$-quotient
$\mathbb{R}\times((S^{2}\times\mathbb{R})/\mathbb{Z}_{2})$ where $\mathbb{Z}_{2}$ flips both $S^{2}$ and $\mathbb{R}$,
or to a quotient of $\mathbb{R}\times S^{3}$ by a finite group acting on the $S^{3}$-factor, and $\hat{g}$ is the standard product metric on these spaces.
Write $\widehat{\nabla}$ for the Levi-Civita connection of $\hat{g}$. What we have then is a sequence of relatively compact open subsets $U_{i}\subset\subset\widehat{M}$ exhausting $\widehat{M}$ and containing $\hat{p}$, together with a sequence of smooth maps $\phi_{i}:U_{i}\to M$ that are diffeomorphisms onto their image, such that $\phi_{i}(\hat{p})=x_{i}$ and
\begin{equation}\label{hot}
|\widehat{\nabla}^{k}(\phi_{i}^{*}g-\hat{g})|_{\hat{g}}\to0\qquad\textrm{for all $k\geq0,$}
\end{equation}
smoothly locally on $\widehat{M}$ as $i\to\infty$. Now, the aforementioned boundedness of $|\operatorname{Rm}(g)|_{g}$ implies that all of the covariant derivatives with respect to $g$ of $\operatorname{Rm}(g)$ are bounded by Shi's derivative estimates. Furthermore, by \cite{naber}, $(M,\,g)$ has a lower bound on its injectivity radius. The conditions of \cite[Theorem 3.22]{chow} are therefore satisfied and consequently we can assert that $(\widehat{M},\,\hat{g})$ is K\"ahler. Since $\RR_{g}(x_{i})>\varepsilon$ for all $i$, $(\widehat{M},\,\hat{g})$ is clearly not flat and so the limit $\mathbb{R}^{4}$ can be discarded. Lifting to the universal cover of the remaining candidates for $\widehat{M}$, we obtain a K\"ahler structure on $\mathbb{R}^{2}\times
S^{2}$ or on $\mathbb{R}\times S^{3}$ with K\"ahler metric we still denote by $\hat{g}$. The following claim allows us to discount the case
$\mathbb{R}\times S^{3}$ next.

\begin{claim}
$\mathbb{R}\times S^{3}$ does not admit a complex structure with respect to which the product metric $\hat{g}$ is K\"ahler.
\end{claim}

\begin{proof}
Suppose to the contrary that $\mathbb{R}\times S^{3}$ admitted a complex structure $J$ with respect to which the product metric $\hat{g}$ is K\"ahler.
Let $q=(x,\,y)\in \mathbb{R}\times S^{3}$. Then we have the decomposition $T_{q}\widehat{M}=T_{x}\mathbb{R}\oplus T_{y}S^{3}$.
Let $Y\in T_{x}\mathbb{R}$ be a unit vector. Then $J$ compatible with $\hat{g}$ implies that\linebreak $\hat{g}(Y,\,JY)$ = 0 which in turn implies that $JY \in T_{y}S^3$.
Let $H$ be an arbitrary holonomy transformation with respect to $\hat{g}$ of $T_{q}\widehat{M}$ and act by $H$ on $JY$.
Then since the holonomy of $\hat{g}$ is trivial on the $\mathbb{R}$-factor and $J$ is parallel,
we see that $H(JY)=J(HY) = JY$, i.e., $H$ fixes $JY$, forcing $JY=0$. This is a contradiction.
\end{proof}

We therefore arrive at the fact that $\widehat{M}$ is covered by $\mathbb{R}^{2}\times S^{2}$. We next identify the $\hat{g}$-compatible complex structure on this space.

\begin{claim}\label{uniquee}
Up to a sign on each factor, the only complex structure on $\mathbb{R}^{2}\times
S^{2}$ with respect to which $\hat{g}$ is K\"ahler is the standard complex structure $\widehat{J}$ on $\mathbb{C}\times\mathbb{P}^{1}$.
\end{claim}

\begin{proof}
Let $q=(x,\,y)\in \mathbb{R}^{2}\times S^{2}\approx\widehat{M}$.
Then we have the decomposition $T_{q}\widehat{M}= T_{x}\mathbb{R}^{2}\oplus T_{y}S^{2}$.
Suppose that there exists another complex structure $\widetilde{J}$ on $\mathbb{R}^{2}\times S^{2}$ with respect to which $\hat{g}$ is K\"ahler and
let $Y\in T_{y}S^{2}$ be a unit vector. Then with respect to the aforementioned decomposition, we can write $\widetilde{J}Y=a\widehat{J}Y\oplus U$ for some $a\in\mathbb{R}$, $|a|\leq 1$,
and $U\in T_{x}\mathbb{R}^{2}$. We parallel transport the quadruple $\{Y,\,\widetilde{J}Y,\,\widehat{J}Y,\,U\}$ around a non-trivial closed loop in the $S^{2}$-fibre of $\widehat{M}$ containing $q$ using the connection $\widehat{\nabla}$. As $\hat{g}$ is flat in the $\mathbb{R}^{2}$-direction, $U$ will remain unchanged under this action. Moreover, $\widetilde{J}$ and $\widehat{J}$ are parallel
with respect to $\hat{g}$. Thus, as the holonomy of $S^{2}$ is $SO(2)$, we find that for every unit vector $Z\in T_{y}S^{2}$, $\widetilde{J}Z=a\widehat{J}Z\oplus U$, leaving us with $U=0$ and $|a|=1$.

We next consider a unit vector $Y\in T_{x}\mathbb{R}^{2}$. Then with respect to the splitting $T_{q}\widehat{M}=T_{x}\mathbb{R}^{2}\oplus T_{y}S^{2}$, we have that
$\widetilde{J}Y=U\oplus b\widehat{J}Y$ for some $b\in\mathbb{R}$, $|b|\leq 1$, and $U\in T_{y}S^{2}$. Arguing as before, parallel transport in the $S^{2}$-fibre of $\widehat{M}$ containing $q$ using $\widehat{\nabla}$ demonstrates that $\widetilde{J}Y=V\oplus b\widehat{J}Y$ for all $V\in T_{y}S^{2}$, forcing $V=0$ and $|b|=1$. From this, the assertion follows.
\end{proof}

Thus, without loss of generality, we may assume that the aforementioned K\"ahler structure on $\mathbb{R}^{2}\times
S^{2}$ is standard, i.e., simply $(\hat{g},\,\widehat{J})$. It then follows that
$\widehat{M}$ is biholomorphic to $\mathbb{C}\times\mathbb{P}^{1}$ as the $\mathbb{Z}_{2}$-quotient thereof, acting freely and holomorphically, would
introduce an $\mathbb{R}P^{2}$ as a complex submanifold yielding a contradiction. Returning to \eqref{hot}, set $J_{i}:=\phi_{i}^{*}J$ and $g_{i}:=\phi^{*}_{i}g$.
Arguing as in \cite[Proof of Theorem 3.22]{chow} (see also \cite[pp.16--18]{ruannn}), we see that $J_{i}$ converges smoothly locally to a $\hat{g}$-parallel complex structure
$J_{\infty}$ on $\widehat{M}$ which by Claim \ref{uniquee}, we can without loss of generality take to be equal to $\widehat{J}$.

Fix a large ball $B_{R}(\hat{p},\,\hat{g})\subset\widehat{M}$ of radius $R>0$ centred at $\hat{p}$ with respect to $\hat{g}$ and let $\hat{u}:\mathbb{P}^{1}\to\mathbb{C}$ denote the unique $\widehat{J}$-holomorphic sphere passing through $\hat{p}$. Since $|J_{i}-\widehat{J}|_{\hat{g}}\to0$ as $i\to\infty$, by Corollary \ref{curve}, for
$i$ sufficiently large, $\hat{u}$ may be deformed to a $J_{i}$-holomorphic sphere
$u:\mathbb{P}^{1}\to\widehat{M}$ with zero self-intersection. By the estimate given in Corollary \ref{curve}, the image of $u:\mathbb{P}^{1}\to\widehat{M}$
will eventually be contained in $B_{R}(\hat{p},\,\hat{g})$. Thus, outside any fixed compact subset $K$ of $M$,
$v:=\phi_{i}^{-1}\circ u:\mathbb{P}^{1}\to M$ will define a $J$-holomorphic curve in $M$ with trivial normal bundle and zero self-intersection lying in $M\setminus K$
for $i$ sufficiently large.

Henceforth we write $C:=v(\mathbb{P}^{1})$. Then $C.C=0$. Recall the real torus $T$ acting on $M$ introduced at the beginning of this section and let $\omega$ denote the K\"ahler form of $g$. The function $f$, the Hamiltonian potential of $JX$, is, as the soliton potential, proper and bounded from below (cf.~Theorem \ref{theo-basic-prop-shrink}). Consequently, Proposition \ref{identify} allows us to find an element $JY\in\Lambda_{\omega}\subseteq\mathfrak{t}$ whose flow generates an $S^{1}$-action and that admits a real Hamiltonian potential $u_{Y}$ that is proper and bounded from below. Since the fixed point set of $T$ is non-empty and contained in the zero set of $X$, a compact subset, \cite[Proposition 1.2]{wu} implies that the zero set of $Y$ is non-empty and compact. Moreover, by \cite[Lemma 2.34]{cds}, we also know that $Y$ and $JY$ are complete. Hence we can define for all time the holomorphic flow of the vector fields $Y$ and $JY$ which we denote by $\phi_{t}^{Y}$ and $\phi_{t}^{JY}$ respectively for $t\in\mathbb{R}$.  As the next claim shows, the image of a holomorphic sphere under the flow of $Y$ and $JY$ is determined by the image of one point on the sphere.

\begin{claim}\label{bollox}
For $x\in M$, let $L_{x}\in[C]$ be a holomorphic sphere in $M$ with $x\in L_{x}$. Then $\phi_{t}^{Y}(L_{x})$
(respectively $\phi_{t}^{JY}(L_{x})$) is the unique holomorphic sphere in $M$ lying in $[C]$ passing through $\phi_{t}^{Y}(x)$ (resp.~$\phi_{t}^{JY}(L_{x})$).
\end{claim}

\begin{proof}
It is clear that the image of $L_{x}$ under the flow of $Y$ and $JY$ is a holomorphic sphere in $M$ lying in $[C]$ passing through
$\phi_{t}^{Y}(x)$ and $\phi_{t}^{JY}(x)$, respectively. No other holomorphic sphere in $[C]$ can pass through these points since $C.C=0$.
\end{proof}

Holomorphic spheres containing a zero of $Y$ are fixed by the flow of $Y$ and $JY$.

\begin{claim}\label{fixxed}
Let $L\in[C]$ be a holomorphic sphere in $M$. Then the following are equivalent.
\begin{enumerate}
  \item $Y$ vanishes at some point $x\in L$.
  \item $\phi_{t}^{Y}(L)=\phi_{t}^{JY}(L)=L$ for all $t$.
  \item $Y$ is tangent to $L$.
\end{enumerate}
\end{claim}

\begin{proof}
\begin{description}
  \item[(i)$\implies$(ii)] By Claim \ref{bollox}, $\phi_{t}^{Y}(L)$ is the unique holomorphic curve in $[C]$ passing through $\phi_{t}^{Y}(x)$. Since
$\phi_{t}^{Y}(x)=x$, we deduce that $\phi_{t}^{Y}(L)=L$.
  \item[(ii)$\implies$(iii)] This is clear.
  \item[(iii)$\implies$(i)] A holomorphic vector field tangent to $\mathbb{P}^{1}$ has at least one zero.
\end{description}
\end{proof}

If $Y$ is nowhere vanishing along the holomorphic sphere, then the image sphere is disjoint from the original.

\begin{claim}\label{nointersect}
Let $L\in[C]$ be a holomorphic sphere in $M$. Then $Y$ is nowhere vanishing on $L$ if and only if there exists $\varepsilon>0$ such that
$\phi_{t}^{Y}(L)\cap L=\emptyset$ and $\phi_{t}^{JY}(L)\cap L=\emptyset$ for all $0<|t|<\varepsilon$.
\end{claim}

\begin{proof}
If $Y$ is nowhere vanishing on $L$, then $Y$ cannot be tangent to $L$ for otherwise it would have a zero along $L$. Thus, $Y$ has a normal component at
some point $x\in L$ so that $\phi_{t}^{Y}(x)\notin L$ for $0<|t|<\varepsilon$ for some $\varepsilon>0$.
By Claim \ref{bollox}, for such values of $t$, $\phi_{t}^{Y}(L)$ will be the unique holomorphic sphere in $[C]$ passing through $\phi_{t}^{Y}(x)$, hence will be
disjoint from $L$. A similar argument applies to $JY$. The converse follows from the implication (i)$\implies$(ii) of Claim \ref{fixxed}.
\end{proof}

As the zero set of $Y$ is compact, by choosing $i$ sufficiently large, we can guarantee that $Y$ is nowhere vanishing along $C$ so that Claim \ref{nointersect} applies with $L=C$.
Henceforth working with the $\mathbb{C}^{*}$-action generated by $Y$ and $JY$, in light of Claim \ref{fixxed}, we then see that $Y$ and $JY$ will be nowhere vanishing on the $\mathbb{C}^{*}$-orbit of $C$ in $M$. Define $\operatorname{Orb}_{\mathbb{C}^{*}}(C):=\{g\cdot C\:|\:g\in\mathbb{C}^{*}\}\subseteq M$.

\begin{claim}\label{free}
There exists a finite cyclic group $\mathbb{Z}_{k}\subset S^{1}\subset\mathbb{C}^{*}$, $k\geq1$, such that the induced action of $\mathbb{C}^{*}/\mathbb{Z}_{k}$ on
$\operatorname{Orb}_{\mathbb{C}^{*}}(C)$ is free.
\end{claim}

\begin{proof}
Claim \ref{bollox} implies that the $\mathbb{C}^{*}$-action on this orbit descends to a (transitive)
$\mathbb{C}^{*}$-action on the holomorphic $\mathbb{P}^{1}$'s in $[C]$ contained in $\operatorname{Orb}_{\mathbb{C}^{*}}(C)$. Claim \ref{nointersect} then implies that this action on the holomorphic $\mathbb{P}^{1}$'s is locally free. Compactness of $S^{1}$ implies that the stabiliser group in $S^{1}\subset\mathbb{C}^{*}$ of $C$ under this action is a finite subgroup of $S^{1}$, hence is a cyclic group of the form $\mathbb{Z}_{k}$ for some $k\geq1$. The induced action of $\mathbb{C}^{*}/\mathbb{Z}_{k}$ on the holomorphic $\mathbb{P}^{1}$'s in $[C]$ contained in
$\operatorname{Orb}_{\mathbb{C}^{*}}(C)$ will therefore be free. Claim \ref{bollox} then tells us that the induced action
of $\mathbb{C}^{*}/\mathbb{Z}_{k}$ on $\operatorname{Orb}_{\mathbb{C}^{*}}(C)$ will be free.
\end{proof}

As $\mathbb{C}^{*}/\mathbb{Z}_{k}\cong\mathbb{C}^{*}$, we may therefore assume without loss of generality that the $\mathbb{C}^{*}$-action generated by $Y$ and $JY$ on $\operatorname{Orb}_{\mathbb{C}^{*}}(C)$ is free. We define a map
\begin{equation}\label{mapp}
\Phi:\mathbb{C}^{*}\times\mathbb{P}^{1}\to M,\qquad (g,\,y)\mapsto g\cdot(v(y)).
\end{equation}
Since the $\mathbb{C}^{*}$-action is free, this defines a biholomorphism onto its image, holomorphic along the $\mathbb{P}^{1}$-direction, and for dimensional reasons demonstrates that for some compact subset $K$ of $M$ containing the zero set of $Y$, $M\setminus K$ is biholomorphic to $\mathbb{P}^{1}\times\mathbb{C}^{*}$. Indeed,
recall that the Hamiltonian potential $u_{Y}:M\to\mathbb{R}$ of $Y$ is proper and bounded from below and that the zero set of $Y$ is compact so that the level sets $u^{-1}_{Y}(\{y\})$ of $u_{Y}$ are
compact and, through the gradient flow of $u_{Y}$, diffeomorphic for all $y>R$ for some $R$ sufficiently large and positive. Hence
with $M$ having only one end \cite[Theorem 0.1]{munteanu}, we obtain a decomposition of the unique end of $M$ as $\bigcup_{y\,\in\,(R,\,+\infty)}u_{Y}^{-1}(\{y\})$. In this picture, one can see that the positive gradient flow of $u_{Y}$, that is, the positive flow of $Y$, moves out to infinity along the unique end of $M$ and from \cite[Proposition 2.28]{cds}, we also read that the negative gradient flow of $u_{Y}$, i.e., the negative flow of $Y$, accumulates in the zero set of $Y$,
a non-empty compact analytic subset of $M$. Thus, the image of $\Phi$ is precisely the complement in $M$ of the minimal compact analytic subset of $M$ containing the zero set of $Y$
and so $M$ fibres as a trivial $\mathbb{P}^{1}$-bundle on the complement of this compact analytic subset. Notice that all the $\mathbb{P}^{1}$-fibres of the fibration are homologous to $C$.

Next, being complete and having bounded scalar curvature, $M$ has finite topological type \cite[Theorem 1.2]{fang}, hence $K$ contains only finitely many $(-1)$-curves. These we blow down
to obtain the minimal model $\varpi:M\to M_{\min}$ of $M$ with complex structure we still denote by $J$. As $M$ is simply connected, $M_{\min}$ will also be simply connected.
Furthermore, the $(-1)$-curves in $K$ are necessarily fixed by the $\mathbb{C}^{*}$-action on $M$ induced by the flow of $Y$ and $JY$
and so the $\mathbb{C}^{*}$-action will extend to a $\mathbb{C}^{*}$-action on $M_{\min}$ in such a way that the map $\varpi:M\to M_{\min}$ is equivariant with respect to these two actions.
The holomorphic vector field $Y$ on $M$ therefore descends to a holomorphic vector field $Y$ on $M_{\min}$ with compact zero set, vanishing at at least the points of $M_{\min}$ that are blown-up to obtain $M$. It is also clear that $\Phi$ induces a biholomorphism from $M_{\min}\setminus\varpi(K)$ to $\mathbb{P}^{1}\times\mathbb{C}^{*}$.
We claim that this $\mathbb{P}^{1}$-fibration at infinity extends in a smooth manner to the interior of $M_{\min}$. This we prove via a continuity argument.

To this end, consider the set $$A:=\{x\in\varpi(K)\:|\:\textrm{$x$ is contained in a holomorphic $\mathbb{P}^{1}$ representing $[C]$}\},$$
where we enlarge $K$ if necessary so that a tubular neighbourhood of its boundary is foliated by $\mathbb{P}^{1}$'s representing $[C]$. Then we have:

\begin{claim}\label{continuity}
$A=\varpi(K)$.
\end{claim}

\begin{proof}
First note that $A$ is non-empty and that the openness of $A$ is immediate from Proposition \ref{local1}. As for closedness, let $x_{i}$ be a sequence of points in $A$ with $x_{i}\to x$ for some $x\in A$. Then for each $i$, there exists a $J$-holomorphic curve $u_{i}:\mathbb{P}^{1}\to M$ passing through $x_{i}$ representing $[C]$. Being contained in the same homology class, these curves all have uniformly bounded area. Therefore by the Gromov compactness theorem \cite{gromov}, there exists a subsequence converging to a tree of $k$ holomorphic $\mathbb{P}^{1}$'s with multiplicity in $[C]$. This limit may be written as $[C]=\Sigma_{i=1}^{k}a_{i}[C_{i}],\,a_{i}>0$. Then
$$0=[C].[C]=\Sigma_{i\neq j}a_{i}a_{j}[C_{i}].[C_{j}]+\Sigma_{i}a_{i}^{2}[C_{i}].[C_{i}].$$
Now, from the equation defining a shrinking K\"ahler-Ricci soliton, we know that for any $J$-holomorphic curve $\hat{C}$ in $M$, $-K_{M}.[\hat{C}]>0$
so that $[\hat{C}].[\hat{C}]\geq-1$ by adjunction. As we are working on $M_{\min}$, this implies that $[C_{i}].[C_{i}]\geq0$ and so $k=1$ and accordingly, the limit is a smooth $\mathbb{P}^{1}$ with multiplicity one. This gives closedness and the claim now follows.
\end{proof}
\noindent Hence we conclude that $M_{\min}$ exhibits the global smooth structure of a $\mathbb{P}^{1}$-fibration over a real surface $S$, with each fibre lying in the homology class $[C]$.

We next holomorphically compactify $M_{\min}$ by adjoining a $\mathbb{P}^{1}$ at infinity using the $\Phi$ from \eqref{mapp} to obtain a closed compact real manifold $\overline{M}_{\min}$ that admits the structure of a smooth $S^{2}$-bundle over a closed compact real surface $\overline{S}$ that itself is obtained from $S$ by adding a point at infinity.
By construction, this additional fibre will be preserved by the induced $\mathbb{C}^{*}$-action on $\overline{M}_{\min}$.
As $M_{\min}$ is simply connected, $\overline{M}_{\min}$ will be simply connected by the Seifert-Van Kampen theorem. It then follows from a long exact sequence
\cite[(17.4)]{Bott} that $\overline{S}$ is simply connected, hence is diffeomorphic to $S^{2}$.
Consequently, $\overline{M}_{\min}$ is diffeomorphic to either $S^{2}\times S^{2}$ or to the blowup of $\mathbb{P}^{2}$ at one point, the only two $S^{2}$-bundles over
$S^{2}$ \cite{steeny}. In either case, removing an $S^{2}$-fibre shows that $S$ is diffeomorphic to $\mathbb{R}^{2}$ and that
$M_{\min}$ is diffeomorphic to $S^{2}\times\mathbb{R}^{2}$ with the $S^{2}$-fibres defining $J$-holomorphic
spheres in $M_{\min}$.

Being a compact analytic subvariety of $M_{\min}$, the zero set of $Y$ must comprise a finite union of isolated points and $\mathbb{P}^{1}$-fibres of $M_{\min}$. Now, those fibres containing a zero of $Y$ are fixed by the $\mathbb{C}^{*}$-action induced by $Y$ and $JY$ by Claim \ref{fixxed}. Otherwise, by Claim \ref{nointersect}, the image of a fibre is disjoint from the original. What we deduce therefore is that the $S^{1}$-action defined by the flow of $JY$ on $M_{\min}$ induces an $S^{1}$-action on $S\approx\mathbb{R}^{2}$ with finitely many zeroes, and in turn via $\Phi$ an $S^{1}$-action on $\overline{S}\approx S^{2}$ with finitely many zeroes, one of which is at infinity. Averaging the round metric on $S^{2}$ over this action, we may assume that the $S^{1}$-action is isometric. \cite[Theorem (4)]{koby} then tells us that the $S^{1}$-action on $\overline{S}$ has precisely two zeroes. As one of these zeroes occurs at infinity, we conclude that the $\mathbb{C}^{*}$-action on $M_{\min}$ fixes precisely one $\mathbb{P}^{1}$-fibre. Denote this fibre by $L_{0}$. By
Claim \ref{fixxed}, $Y$ is then tangent to $L_{0}$ and by Claim \ref{nointersect}, the zero set of $Y$ is contained in $L_{0}$. As the flow of $JY$ induces an $S^{1}$-action on $L_{0}$, we again see from \cite{koby} that the zero set of $Y$ comprises the whole of $L_{0}$ (if the $S^{1}$-action is trivial) or precisely two points. As $M$ is obtained from $M_{\min}$
by blowing up finitely many points of $M_{\min}$ at which the vector field $Y$ vanishes, we see that $M$ is obtained from $M_{\min}$ by blowing up finitely many points of
$L_{0}$. Blowing up more than one point would introduce at least one holomorphic sphere in $M$ with self-intersection $(-k)$ for some $k\geq2$. This is not possible because using adjunction,
the restriction of $-K_{M}$ to every holomorphic curve in $M$ must be positive by the shrinking soliton condition.
Hence $\varpi:M\to M_{\min}$ is the identity or the blowup of $M_{\min}$ at one point of $L_{0}$. Set $E:=\varpi^{-1}(L_{0})$. Then $E$ contains the zero set of $Y$ on $M$ and hence also the fixed point set of $T$. But the flow of $JX$, being dense in $T$, implies that this latter set coincides with the zero set of $X$. This completes the proof of the proposition.
\end{proof}

We now consider $\widehat{M}:=\mathbb{C}\times\mathbb{P}^{1}$ endowed with the standard holomorphic action of the real two-dimensional torus $\widehat{T}$ with Lie algebra $\hat{\mathfrak{t}}$
and $\widetilde{M}:=\operatorname{Bl}_{p}(\mathbb{C}\times\mathbb{P}^{1})$, the blowup of $\widehat{M}$ at a fixed point $p$ of the $\widehat{T}$-action on $\widehat{M}$. The torus action on $\widehat{M}$ induces in a natural way the holomorphic action of a real two-dimensional torus $\widetilde{T}$ on $\widetilde{M}$ with Lie algebra $\tilde{\mathfrak{t}}$ such that the blowdown map $\sigma:\widetilde{M}\to\widehat{M}$ is equivariant with respect to the action of $\widetilde{T}$ and $\widehat{T}$.
Recall the real torus $T$ generated by the flow of $JX$ with Lie algebra $\mathfrak{t}$ containing $JX$
acting on $(M,\,J,\,g)$ in a holomorphic isometric fashion with a compact fixed point set introduced at the beginning of this section.
Theorem \ref{mainthm1}(i) will follow from the next proposition, an improvement from the smooth category of the previous proposition to the complex category.

\begin{prop}[Holomorphic classification]\label{holoo}
Let $(M,\,g,\,X)$ be a two-dimensional simply connected complete non-compact shrinking gradient K\"ahler-Ricci soliton with
bounded scalar curvature $\RR_{g}$ with $X=\nabla^{g}f$ for some smooth function
$f:M\rightarrow\mathbb{R}$. Assume that $X$ has an integral curve along which $\RR_{g}\not\to0$.
Then there exists an equivariant biholomorphism $\alpha$ from $(M,\,T)$ to $(\widehat{M},\,\widehat{T})$ or $(\widetilde{M},\,\widetilde{T})$ with respect to which $\alpha_{*}(JX)$ lies in $\hat{\mathfrak{t}}$ or $\tilde{\mathfrak{t}}$, respectively. In particular, in the latter case, $\alpha_{*}(JX)$ is given by \eqref{vfield}.
\end{prop}

\begin{proof}
We have already established in Proposition \ref{smoothclass} that $M$ is diffeomorphic to either $\mathbb{C}\times\mathbb{P}^{1}$ or to $\operatorname{Bl}_{p}(\mathbb{C}\times\mathbb{P}^{1})$ and that there is a map $\varpi:M\to M_{\min}$ to $M_{\min}$, a manifold diffeomorphic to $S^{2}\times\mathbb{R}^{2}$ with the $S^{2}$-fibres defining holomorphic spheres in $M_{\min}$, with $\varpi$ the identity or the blowup of $M_{\min}$ at
a point $p$ of a $\mathbb{P}^{1}$-fibre $L_{0}$ of $M_{\min}$, as appropriate. Let $\hat{\pi}:M_{\min}\to\mathbb{R}^{2}$ denote the projection map. Then we obtain a map $\pi:=\hat{\pi}\circ\varpi:M\to\mathbb{R}^{2}$. Without loss of generality, we may assume that $L_{0}=\hat{\pi}^{-1}(\{0\})$.
Proposition \ref{smoothclass} then tells us that $M_{0}(X)$, that is, the zero set of $X$, is a compact analytic subset of
$E:=\pi^{-1}(\{0\})$, where $E$ is equal to $L_{0}$ if $\varpi$ is the identity, or to two holomorphic $\mathbb{P}^{1}$'s meeting transversely, each of self-intersection $(-1)$, otherwise. In this latter case, we denote these curves by $L_{1}$ and $L_{2}$. In both cases, the forward flow of $-X$ accumulates in $E$ by \cite[Proposition 2.28]{cds} and the action of $T$ preserves $E$. Indeed, this last point follows from Claims \ref{bollox} and \ref{fixxed} (which also hold with $Y$ replaced by $X$) if $\varpi$ is the identity map, and from the following claim otherwise.

\begin{claim}\label{ireland}
$X$ is tangent to any $(-1)$-curve in $M$.
\end{claim}

\begin{proof}
A neighbourhood of any $(-1)$-curve in $M$ is biholomorphic to a neighbourhood of the zero section of $\mathcal{O}_{\mathbb{P}^{1}}(-1)$. Along this zero section, we have a canonical
holomorphic splitting of $TM$ as $T\mathbb{P}^{1}\oplus\mathcal{O}_{\mathbb{P}^{1}}(-1)$. The normal component of $X$ in this splitting must therefore vanish which yields the claim.
\end{proof}

This claim in fact implies that when $\varpi$ is the blowup map, $L_{1}$ and $L_{2}$ are both preserved by the action of $T$. Thus, no matter what $\varpi$ may be,
the action of $T$ on $M$ will induce an action of $T$ on $M_{\min}$. In the particular case when $\varpi$ is the blowup map,
the point of intersection of $L_{1}$ and $L_{2}$ will be fixed by $T$. We denote this point by $x$ so that $x\in M_{0}(X)\cap E\neq\emptyset$.

Suppose first that $\varpi$ is the blowup map so that
$M$ is diffeomorphic to $\operatorname{Bl}_{p}(\mathbb{C}\times\mathbb{P}^{1})$. We begin by noting:
\begin{claim}\label{spring}
If $X|_{L_{i}}$ is non-trivial for $i=1$ or $i=2$, then $|M_{0}(X)\cap L_{i}|=2$.
\end{claim}

\begin{proof}
As $X|_{L_{i}}$ is non-trivial, the restriction of $f$ to $L_{i}$ is non-constant, hence attains a global maximum and a global minimum on $L_{i}$. At these points, $d(f|_{L_{i}})=0$. Then as $X$ is tangent to $L_{i}$ by Claim \ref{ireland}, we actually have that $df=0$ at these points so that $X|_{L_{i}}$ has at least two zeroes. But $X|_{L_{i}}$ is a holomorphic vector field on $\mathbb{P}^{1}$, hence has at most two zeroes.
\end{proof}

Now, by \cite[Proof of Lemma 1]{frankel}, $f$ is a Morse-Bott function on $M$. The critical submanifolds of $f$
are precisely the connected components of $M_{0}(X)$. Since $M$ is K\"ahler,
the Morse indices (i.e., the number of negative eigenvalues of $\operatorname{Hess}(f)$) of the critical submanifolds are all even \cite{frankel}.
Write $$M_{0}(X)=M^{(0)}\cup M^{(2)}\cup M^{(4)},$$
where $M^{(j)}$ denotes the disjoint union of the critical submanifolds of $M_{0}(X)$ of index $j$.
We already know from Proposition \ref{smoothclass} that $M_{0}(X)\subseteq E$,
and from \cite[Claim 2.30]{cds} we know that $M^{(0)}$ is a non-empty, connected, compact complex submanifold of $M$, hence
is equal to either $E$, $L_{1}$, $L_{2}$, or an isolated point of $E$. We analyse the structure of $M_{0}(X)$ in each of these cases separately, beginning with:

\begin{claim}\label{winter}
If $M^{(0)}$ comprises a single point, then $|M_{0}(X)\cap L_{i}|=2$ for $i=1,\,2$.
\end{claim}

\begin{proof}
Recall that $\{x\}=L_{1}\cap L_{2}$ and assume that $M^{(0)}=\{y\}$ for some point $y\in E$.
If $x=y$, then $y$ is an isolated zero of both $X|_{L_{1}}$ and $X|_{L_{2}}$. The result then follows by applying Claim \ref{spring}
to both $L_{1}$ and $L_{2}$. If $x\neq y$, then assume without loss of generality that $y\in L_{1}$.
Then the zero set of $X|_{L_{1}}$ comprises at least two points, namely $x$ and $y$. Claim \ref{spring} implies that in fact
$|M_{0}(X)\cap L_{1}|=\{x,\,y\}$. Considering $X|_{L_{2}}$ next, the zero set of this vector field contains $x$.
If $X|_{L_{2}}$ is non-trivial, then the result follows from Claim \ref{spring}. Otherwise assume that
$X|_{L_{2}}\equiv0$. Then as $M^{(0)}$ is connected, we have that $x\in M^{(2)}\cup M^{(4)}$. Now, if $x\in M^{(4)}$, then by
\cite[Proposition 6]{Bry-Kah-Sol} there exist local holomorphic coordinates $(z_{1},\,z_{2})$ centred at $x$ such that
the holomorphic vector field $X^{1,\,0}:=\frac{1}{2}(X-iJX)$ takes the form
\begin{equation*}
X^{1,\,0}=a_{1}z_{1}\frac{\partial}{\partial z_{1}}+a_{2}z_{2}\frac{\partial}{\partial z_{2}}
\end{equation*}
for some $a_{1},\,a_{2}\in\mathbb{R}_{<\,0}$. This implies in particular that $x$ is an isolated zero of $X$, contradicting the fact that
$X|_{L_{2}}\equiv0$. Hence necessarily $x\in M^{(2)}$ so that $L_{2}\subseteq M^{(2)}$.

For each $z\in L_{2}$, $f$ is decreasing along the forward flow of $-X$ emanating from $z$, hence this flow accumulates at $y\in L_{1}$
by \cite[Proposition 2.8]{cds}. As in \cite[p.3332]{chen-soliton}, we can use the forward flow of $-X$ to construct a holomorphic sphere $R_{z}:\mathbb{P}^{1}\to M$ in $M$ with $R_{z}(0)=z$ and $R_{z}(\infty)=y$. Assume that $x\neq z$ and call the resulting holomorphic sphere $D$. Then for $i\neq j$, $L_{j}=\varpi^{*}C-L_{i}$, $D.L_{i}>0$, and $D.L_{j}>0$, which leads to the conclusion
that $D.\varpi^{*}C>0$. This is a contradiction and the claim now follows.
\end{proof}

Next, we have:
\begin{claim}\label{summer}
If $M^{(0)}=L_{i}$, then $|M_{0}(X)\cap L_{j}|=2$ for $j\neq i$.
\end{claim}

\begin{proof}
In this case, $X|_{L_{j}}$ is non-trivial as $M^{(0)}\neq E$ and is connected. The result then follows from an application of Claim \ref{spring} to
$L_{j}$.
\end{proof}

Thus, the induced action of the real torus $T$ on $M_{\min}$ will fix either two points on $L_{0}$ or the whole of $L_{0}$ and the forward flow lines of the vector field $-X$ induced on $M_{\min}$ accumulate in $L_{0}$. If $\dim_{\mathbb{R}}T=2$, then by identifying a point off of $L_{0}$ and $\{0\}\times\mathbb{P}^{1}$ in $M$ and $\widehat{M}$ respectively and using the actions, one can construct an equivariant biholomorphism
$\alpha:(M_{\min},\,T)\to(\widehat{M},\,\widehat{T})$.

If $|M^{(0)}|=1$ and $\dim_{\mathbb{R}}T=1$, then the flow of $X$ and $JX$ on $M_{\min}$ induces a $\mathbb{C}^{*}$-action on $M_{\min}$
which by Claim \ref{free} we may assume to be free on $M_{\min}\setminus L_{0}$. In addition,
Claim \ref{winter} implies that the fixed point set of $T$ will comprise precisely two isolated points in $L_{0}$, $a$ and $b$ say.
As the forward flow lines of $-X$ on $M_{\min}$ accumulate in $L_{0}$, the closure of every orbit of this $\mathbb{C}^{*}$-action on $M_{\min}$ is a copy of $\mathbb{C}$ obtained by adjoining either $a$ or $b$ to the orbit in question. Choose an orbit $O_{a}$ and $O_{b}$ passing through $a$ and $b$, respectively. Then each orbit will intersect every fibre of the $\mathbb{P}^{1}$-foliation of $M_{\min}\setminus L_{0}$ at precisely one point. Indeed, if an orbit intersected a $\mathbb{P}^{1}$-fibre $L$ in $M_{\min}\setminus L_{0}$ at two points $x_{1},\,x_{2}\in L,\,x_{1}\neq x_{2},$ say, then there would exist a $g\in\mathbb{C}^{*},\,g\neq1,$ such that $g\cdot x_{1}=x_{2}$. By Claim \ref{bollox}, we would then have that $g\cdot L=L$. The element $g$ would then define an automorphism of $L\cong\mathbb{P}^{1}$ and would therefore have at least one fixed point. This contradicts the freeness of the $\mathbb{C}^{*}$-action on $M_{\min}\setminus L_{0}$.
Define a global real holomorphic vector field $JV$ on $M_{\min}$ in the following way. Restricted to a $\mathbb{P}^{1}$-fibre $L$, $JV$ will be the unique real holomorphic vector field tangent to $L$ vanishing at $O_{a}\cap L$ and $O_{b}\cap L$
that generates a holomorphic $S^{1}$-action, the direction of the flow of which
relative to the points $O_{a}\cap L$ and $O_{b}\cap L$ will be the same as that on $L_{0}$ relative to $a$ and $b$,
and the time $2\pi$-flow of which is the identity map. The flow of $V$ and $JV$ will generate a $\mathbb{C}^{*}$-action on $M_{\min}$ that commutes with the flow of $X$ and $JX$. To see this last point, it suffices to verify that $[X,\,V]=0$ on $M_{\min}\setminus L_{0}$. To this end, we set up an equivariant biholomorphism $\mathbb{C}^{*}\times\mathbb{P}^{1}\to M_{\min}\setminus L_{0}$ in the following way.
Pick an arbitrary fibre $u:\mathbb{P}^{1}\to L\subset M_{\min}\setminus L_{0}$. By pre-composing with a suitable M\"obius transformation, we can assume that $u(0)=O_{a}\cap L$ and
$u(\infty)=O_{b}\cap L$. As in \eqref{mapp}, we extend $u$ to a biholomorphism $\Psi:\mathbb{C}^{*}\times\mathbb{P}^{1}\to M_{\min}\setminus L_{0}$, equivariant
with respect to the standard $\mathbb{C}^{*}$-action on the first component of the domain and the $\mathbb{C}^{*}$-action generated by $X$ and $JX$ on the range.
By construction, $\Psi^{-1}$ has the property that it pushes forward $\frac{1}{2}(V+iJV)$ to a global holomorphic vector field on $\mathbb{C}^{*}\times\mathbb{P}^{1}$ tangent to the $\mathbb{P}^{1}$-fibres and vanishing along $(\mathbb{C}^{*}\times\{0\})\cup(\mathbb{C}^{*}\times\{\infty\})$. In particular, this holomorphic vector field generates another $\mathbb{C}^{*}$-action
on $\mathbb{C}^{*}\times\mathbb{P}^{1}$ and the map $\Psi$ will also be $\mathbb{C}^{*}$-equivariant with respect to this action on the domain and that generated by
$\frac{1}{2}(V+iJV)$ on the range. Observing that the two $\mathbb{C}^{*}$-actions on the domain of $\Psi$ commute, the desired vanishing of $[X,\,V]$ is now clear.
The result of this is that $T$ is contained in a real two-dimensional torus acting holomorphically on $M_{\min}$ and hence we reduce to the previous case.

If $M^{(0)}=L_{i}$ for some $i=1,\,2,$ then by Claim \ref{summer}, the fixed point set of $\widehat{T}$ will comprise either two isolated points, a case that we have already dealt with (independent of the dimension of $T$), or a $\mathbb{P}^{1}$ given by $\varpi(M^{(0)})$. In this latter case, the argument of the proof of \cite[Claim 4.15]{charlie}
tells us that $\dim_{\mathbb{R}}T=1$ (this argument is local). The argument of \cite[Claims 4.16 and 4.17]{charlie} then yields an equivariant biholomorphism $\alpha:(M_{\min},\,T)\to(\widehat{M},\,\widehat{T})$.

Finally, if $M^{(0)}=E$, then the fixed point set of $\widehat{T}$ will comprise a $\mathbb{P}^{1}$ given by $\varpi(M^{(0)})$, that is, an instance of the previous case.
This covers all possibilities for $\varpi$ equal to the blowdown map and so we have an equivariant biholomorphism
$\alpha:(M_{\min},\,T)\to(\widehat{M},\,\widehat{T})$. Being equivariant then allows us to lift this to an equivariant biholomorphism $\alpha:(M,\,T)\to(\widetilde{M},\,\widetilde{T})$.

Suppose now that $\varpi$ is the identity map so that $M$ is diffeomorphic to $\mathbb{C}\times\mathbb{P}^{1}$ and $M=M_{\min}$. Then $E=\mathbb{P}^{1}$ and is preserved by the action of
$T$. As $M_{0}(X)\subseteq E$, we must therefore have that the fixed point set of $T$ comprises either two points in $E$ or the whole of $E$. Being connected, it follows that
$|M^{(0)}|=1$ in the former case and that $M^{(0)}=E$ in the latter case. All possibilities thereafter have then been dealt with above
and we conclude that there is an equivariant biholomorphism $\alpha:(M,\,T)\to(\widehat{M},\,\widehat{T})$.

In both cases, the fact that $JX$ generates $T$ and $\alpha$ is equivariant implies that $\alpha_{*}(JX)\in\tilde{\mathfrak{t}}$ or $\alpha_{*}(JX)\in\hat{\mathfrak{t}}$, as appropriate.
\end{proof}

We now conclude the proof of Theorem \ref{mainthm1}.

\subsection*{Completion of the proof of Theorem \ref{mainthm1}}
Given $(M,\,g,\,X)$ as in the statement of Theorem \ref{mainthm1}, let $M_{\textrm{univ}}$ denote the universal cover
of $M$. Then since $M$ has finite fundamental group \cite[Theorem 1.1]{wyliee}, we can write
$M=M_{\textrm{univ}}/\Gamma$, where $\Gamma$ is a finite group of biholomorphisms of $M_{\textrm{univ}}$ acting freely.
Lifting the shrinking soliton structure to $M_{\textrm{univ}}$, we read from Proposition \ref{holoo} that $M_{\textrm{univ}}$
is biholomorphic to either $\widehat{M}$ or $\widetilde{M}$. Thus, item (i) of Theorem \ref{mainthm1} will follow from Proposition \ref{holoo}
if we can establish that $\operatorname{Fix}(\Gamma)\neq\emptyset$. This we prove in the next claim.

\begin{claim}
Every element of a finite group $\Gamma$ of biholomorphisms acting on $\widehat{M}$ or $\widetilde{M}$ has a fixed point.
\end{claim}

\begin{proof}
Any biholomorphism of $\widetilde{M}$ must preserve the two $(-1)$-curves, hence it must fix their point of intersection $x$.

As for $\widehat{M}$, any automorphism $\gamma\in\Gamma$ sends a $\mathbb{P}^{1}$-fibre to a $\mathbb{P}^{1}$-fibre,
hence $\gamma$ induces an automorphism of the $\mathbb{C}$-factor of $\widehat{M}$. Every finite automorphism group of
$\mathbb{C}$ is a rotational group. In particular, the origin is fixed by the action, and so there exists a $\mathbb{P}^{1}$-fibre
of $\widehat{M}$ fixed by $\Gamma$. Every M\"obius transformation has a fixed point. This observation completes the proof of the claim.
\end{proof}

The biholomorphism $\alpha:M\to M$ given by Proposition \ref{holoo} has the property that $\alpha_{*}(JX)$ lies in
$\operatorname{Lie}(\mathbb{T})$, the Lie algebra of the real torus $\mathbb{T}$ from Theorem \ref{mainthm1}(ii).
Let $X':=\alpha_{*}(X)$, $g':=(\alpha^{-1})^{*}g$, and consider the complete shrinking soliton
$(M,\,X',\,g')$. The fact that $\alpha$ is a biholomorphism implies that the background complex structure here is still $J$.
In particular, $JX'\in\operatorname{Lie}(\mathbb{T})$.

Let $G^{X'}_{0}$ denote the connected component of the identity of the holomorphic isometries of $(M,\,J,\,g')$ that
commute with the flow of $X'$. As explained at the beginning of Section \ref{proof1},
the assumption of bounded scalar curvature implies that the closure of the flow of $JX'$ in $G^{X'}_{0}$ yields the holomorphic isometric action of a real torus $T'$ on $(M,\,J,\,g')$ with Lie algebra $\mathfrak{t}'$ containing $JX'$. Without loss of generality, we may assume that $T'$ is maximal in $G^{X'}_{0}$.
\cite[Corollary 5.13]{cds} asserts that $G^{X'}_{0}$ is a maximal compact Lie subgroup of the Lie group $\operatorname{Aut}^{X'}_{0}(M)$, the connected component of the identity of the group of automorphisms of $(M,\,J)$ that commute with the flow of $X'$; cf.~\cite[Proposition 5.8]{cds} as for
why $\operatorname{Aut}^{X'}_{0}(M)$ is a Lie group. Thus, $T'$ is a maximal real torus in $\operatorname{Aut}^{X'}_{0}(M)$.
For each $v\in\operatorname{Lie}(\mathbb{T})$, $JX'\in\operatorname{Lie}(\mathbb{T})$ implies that $[v,\,JX']=0$ so that $[v,\,X']=0$. Hence each element of $\mathbb{T}$ commutes with the flow of $X'$ and so $\mathbb{T}$ itself is a Lie subgroup of $\operatorname{Aut}^{X'}_{0}(M)$. For dimensional reasons $\mathbb{T}$ is maximal in $\operatorname{Aut}^{X'}_{0}(M)$, therefore by Iwasawa's
theorem \cite{iwasawa} there exists an element $\beta\in\operatorname{Aut}^{X'}_{0}(M)$ such that
$\beta(T')\beta^{-1}=\mathbb{T}$. Since $\beta$ commutes with the flow of $X'$, necessarily $d\beta^{-1}(X')=X'$. Moreover, $\beta^{*}(g')$ is invariant under the action of $\mathbb{T}$.
Let $\gamma:=\alpha^{-1}\circ\beta:M\to M$. Unravelling the definitions, we conclude that
$\gamma^{*}g$ is invariant under the action of $\mathbb{T}$ and $\gamma_{*}^{-1}(JX)=JX'\in\operatorname{Lie}(\mathbb{T})$.
This yields item (ii) of Theorem \ref{mainthm1}. Note that the background complex structure $\gamma^{*}J$ is still equal to $J$ because $\gamma$ is a biholomorphism.

The fact that $\gamma_{*}^{-1}(JX)$ is determined in item (iii) is a result of Proposition \ref{properr} and Theorem \ref{thmB13}, as we know for this latter theorem that the Ricci curvature of $g$, hence that of $\gamma^{*}g$, is bounded. That its flow generates an $S^{1}$-action is clear from the explicit expression of the vector field, given in Examples \ref{vf-id2} and \ref{vf-id} for each respective possibility of $M$. As explained at the beginning of this section, $JX$ is holomorphic and Killing and so the flow of $\gamma_{*}^{-1}(JX)$ is holomorphic and isometric for $(J,\,\gamma^{*}g)$, as claimed in the same item.

Finally, item (iv) follows from the toricity of the soliton and an application of \cite[Theorem A]{charlie}.

\section{Proof of Theorem \ref{thmc}}\label{proof2}

Recall that $(M,\,g(t))$ is a finite time Type I K\"ahler-Ricci flow on $[0,\,T),\,T<+\infty$, defined on a compact K\"ahler surface $M$,
$x\in\Sigma_{I}\subset M$ is a Type I singular point, and $g_{j}(t):=\lambda_{j}g(T+\nolinebreak\frac{t}{\lambda_{j}}),\,t\in[-\lambda_{j}T,\,0),$ for a
sequence $\lambda_{j}\to+\infty$. Let $J$ denote the complex structure of $M$. From \cite{topping, naber}, we know that a subsequence
of $(M,\,g_{j}(t),\,x)$ converges in the smooth pointed Cheeger-Gromov sense \cite[Definition 7.2.1]{topping-book}
to a non-flat complete shrinking gradient Ricci soliton $(N,\,h,\,p)$ with bounded curvature and soliton potential $f$ and associated K\"ahler-Ricci flow
$h(t),\,t\in(-\infty,\,0),$ with $h(-1)=h$. Uniformly bounded curvature implies from Shi's derivative estimates that
the norm of the derivatives of the curvatures of
the metrics $g_{j}(t)$ are uniformly bounded, hence an application of
\cite[Theorem 3.23]{chow} demonstrates that the limit is in fact K\"ahler so that $(N,\,h)$ is a two-dimensional shrinking gradient K\"ahler-Ricci soliton with bounded scalar curvature.
Let $\widetilde{J}$ denote the complex structure of $N$.

First assume that $\lim_{t\to T^{-}}\vol_{g(t)}(M)>0$. Then if $N$ were compact, $N$ would be a del Pezzo surface with $h$ K\"ahler-Einstein
or the shrinking gradient K\"ahler-Ricci soliton on the blowup of $\mathbb{P}^{2}$ at one or two points \cite{soliton}.
After unravelling the scaling factors in the definition of smooth pointed Cheeger-Gromov convergence, this would then imply that
$\lim_{t\to T^{-}}\operatorname{vol}_{g(t)}(M)=0$, a contradiction. Indeed,
let $h(t),\,t\in(-\infty,\,0),$ denote the K\"ahler-Ricci flow associated to $(N,\,h)$.
Then compactness of $N$ implies that for all
$0<\delta<1$, there exists a diffeomorphism $\phi_{k}:N\to M$ such that\linebreak $|\phi_{k}^{*}g_{k}(t)-h(t)|<1$
with derivatives for all $t\in[-1,\,-\delta]$ for $k$ sufficiently large. In particular,
 \begin{equation*}
\begin{split}
\operatorname{vol}_{g\left(T+\frac{t}{\lambda_{k}}\right)}(M)&=\frac{\operatorname{vol}_{\phi_{k}^{*}g_{k}(t)}(N)}{\lambda_{k}^2}
\leq\frac{C}{\lambda_{k}^2}\to0\qquad\textrm{as $k\to\infty$}.\\
\end{split}
\end{equation*}
Thus, $(N,\,h)$ is non-compact and according to Lemma \ref{batman}, the scalar curvature $\RR_{h}$ of $h$ tends to zero along the unique end
of $N$ or there exists an integral curve of the soliton vector field of $(N,\,h)$ along which $\RR_{h}\not\to0$.
If $\RR_{h}\to0$, then we would be done by \cite[Theorem E(3)]{cds}. Therefore to conclude the proof of this direction of the theorem, it suffices to rule out the latter case.

To this end, recall from Theorem \ref{mainthm1} that if
there exists an integral curve of the soliton vector field of $(N,\,h)$ along which $\RR_{h}\not\to0$, then up to pullback by biholomorphism, $(N,\,h)$ is the cylinder
$\mathbb{C}\times\mathbb{P}^{1}$ or a hypothetical shrinking K\"ahler-Ricci soliton on the blowup of $\mathbb{C}\times\mathbb{P}^{1}$
at one point. In either case, choose $R>0$ such that $p\in f^{-1}((-\infty,\,3R])$,
$A:=f^{-1}([2R,\,3R])$ is a non-empty annulus in $N$, and such that any $(-1)$-curves are contained in the
set $f^{-1}((-\infty,\,R])$. This can be done because $f$ is proper (cf.~Theorem \ref{theo-basic-prop-shrink}). Then there exists a compact subset $U\subset N$ containing $A$, $\delta\in(0,\,1)$, and diffeomorphisms $\phi_{k}:U\to M$ with $\phi_{k}(p)=x$ such that $\phi_{k}^{*}g_{k}(t)\to h(t)$ with derivatives on $U$ as $k\to\infty$ for all $t\in[-1,\,-\delta]$.

Next, fix a $\widetilde{J}$-holomorphic sphere $u:\mathbb{P}^{1}\to A$ in $N$ with trivial self-intersection. Then by Corollary \ref{curve}, there exists a sequence of
$\phi_{k}^{*}J$-holomorphic spheres $u_{k}:\mathbb{P}^{1}\to N$ with trivial self-intersection converging in $C^{0}$ to $u$ as $k\to\infty$.
In fact, it follows from standard elliptic bootstrapping arguments that there exists a subsequence, still denoted by $u_{k}$, that converges
uniformly with all derivatives to $u$; cf.~\cite[Proposition 3.3.5 and Section B.4]{dusa2}. Set $C:=u(\mathbb{P}^{1})$, $C_{k}:=u_{k}(\mathbb{P}^{1})$, and let $\tau(t),\,t\in(-\infty,\,0),$ denote the K\"ahler form associated to $h(t)$ with associated Ricci form $\rho_{\tau(t)}$. Then (cf.~\cite[equation (1.13)]{chow})
$$[\rho_{\tau(t)}]=\frac{[\tau(t)]}{-t},\qquad t<0.$$ Consequently, using adjunction, we find that for $t<0$,
\begin{equation*}
\begin{split}
\operatorname{vol}_{h(t)}(C)&=\int_{C}[\tau(t)|_{C}]=-t\int_{C}[\rho_{\tau(t)}|_{C}]=-2\pi t\int_{C}c_{1}(-K_{N}|_{C})=-4\pi t.
\end{split}
\end{equation*}
Since $\phi_{k}^{*}g_{k}(t)\to h(t)$ and $u_{k}\to u$ in $C^{1}$ as $k\to\infty$, we can assert that for $t\in[-1,\,-\delta]$,
$u_{k}^{*}(\vol_{\phi_{k}^{*}g_{k}(t)})\to u^{*}(\operatorname{vol}_{h(t)})$ on $\mathbb{P}^{1}$ as $k\to\infty$,
so that for $t\in[-1,\,-\delta]$,
$$\operatorname{vol}_{\phi_{k}^{*}g_{k}(t)}(C_{k})\to\operatorname{vol}_{h(t)}(C)=-4\pi t\qquad\textrm{as $k\to\infty$}.$$
In other words, for $t\in[-1,\,-\delta]$,
\begin{equation}\label{limitt}
\left|\operatorname{vol}_{\phi_{k}^{*}g_{k}(t)}(C_{k})-(-4\pi t)\right|\to0\qquad\textrm{as $k\to\infty$.}
\end{equation}
On the other hand, let $\omega(t)$ denote the K\"ahler form of $g(t)$ and $\rho_{\omega(t)}$ the corresponding Ricci form. Then
$\frac{\partial\omega(s)}{\partial s}=-\rho_{\omega(s)},\,s\in[0,\,T),$ implies that $[\omega(s)]=[\omega(0)]-s[\rho_{\omega(0)}]$. Using this
and the fact that $\phi_{k}(C_{k})$ is $J$-holomorphic, we compute that
\begin{equation}\label{contradiction}
\begin{split}
\operatorname{vol}_{\phi_{k}^{*}g_{k}(t)}(C_{k})&=\operatorname{vol}_{\lambda_{k}g\left(T+\frac{t}{\lambda_{k}}\right)}(\phi_{k}(C_{k}))
=\lambda_{k}\operatorname{vol}_{g\left(T+\frac{t}{\lambda_{k}}\right)}(\phi_{k}(C_{k}))\\
&=\lambda_{k}\int_{\phi_{k}(C_{k})}\left[\left.\omega\left(T+\frac{t}{\lambda_{k}}\right)\right|_{\phi_{k}(C_{k})}\right]\\
&=\lambda_{k}\int_{\phi_{k}(C_{k})}\left([\omega(0)]-2\pi\left(T+\frac{t}{\lambda_{k}}\right)c_{1}\left(-K_{M}|_{\phi_{k}(C_{k})}\right)\right)\\
&=\lambda_{k}\int_{\phi_{k}(C_{k})}\left([\omega(0)]-2\pi Tc_{1}\left(-K_{M}|_{\phi_{k}(C_{k})}\right)\right)-2\pi t\int_{C_{k}}c_{1}\left(-K_{M}|_{\phi_{k}(C_{k})}\right)\\
&=\lambda_{k}\lim_{s\to T^{-}}\operatorname{vol}_{g(s)}(\phi_{k}(C_{k}))-4\pi t.\\
\end{split}
\end{equation}

To derive a contradiction, we need to show that $\lim_{s\to T^{-}}\operatorname{vol}_{g(s)}(\phi_{k}(C_{k}))>c$ for some positive constant $c$ independent of $k$.
For this, we require:

\begin{claim}\label{nhd}
There exists an open subset $U\subset M$ such that for all $k$, $\phi_{k}(C_{k})\cap (M\setminus U)\neq\emptyset$.
\end{claim}

\begin{proof}[Proof of Claim \ref{nhd}]
Let $U$ denote the union of the maximal open neighbourhood of each $(-1)$-curve in $M$ for which there exists a biholomorphism onto
a neighbourhood of the zero section of the line bundle $\mathcal{O}_{\mathbb{P}^{1}}(-1)\to\mathbb{P}^{1}$. Then for every $k$, $\phi_{k}(C_{k})$ has trivial self-intersection in $U$ and so
cannot be contained in $U$ for any $k$. In other words, $\phi_{k}(C_{k})\cap (M\setminus U)\neq\emptyset$ as claimed.
\end{proof}

Let $V$ be an open subset of $M$ containing every $(-1)$-curve in $M$ with $\overline{V}\subset U$. Since \linebreak $\lim_{s\to T^{-}}\operatorname{vol}_{g(s)}(M)>0$ by assumption, we read from
\cite[Theorem 3.8.3]{Bou-Eys-Gue} (cf.~also \cite[Definition 3.7.9]{Bou-Eys-Gue}) that as $t\to T^{-}$, $g(t)$ contracts only $(-1)$-curves
and converges smoothly locally to a K\"ahler metric $g_{T}$ on the complement of these curves. In particular,
$g(t)\to g_{T}$ smoothly on $M\setminus V$ as $t\to T^{-}$ so that $\operatorname{inj}_{M\setminus U}g(t)\to\operatorname{inj}_{M\setminus U}g_{T}$
and $\operatorname{dist}_{g(t)}(\partial U,\,\partial V)\to\operatorname{dist}_{g_T}(\partial U,\,\partial V)$ as $t\to T^{-}$.
Moreover, by the previous claim, for every $k$ there exists a point
$x_{k}\in\phi_{k}(C_{k})\cap (M\setminus U)$. Let $\varepsilon:=\min\{\operatorname{dist}_{g_T}(\partial U,\,\partial V),\,\operatorname{inj}_{M\setminus U}g_{T}\}$.
Then for $s\in(0,\,T)$ sufficiently close to $T$, $B_{g(s)}\left(x_{k},\,\frac{\varepsilon}{2}\right)$ is contained in $M\setminus V$ and an application of \cite[Comment 1, p.178,
and Proposition 4.3.1(ii)]{sikorav} (see also \cite[Lemma 5.2]{collins100}) yields for such values of $s$ the lower bound
\begin{equation*}\label{volume}
\operatorname{vol}_{g(s)}(\phi_{k}(C_{k}))\geq\operatorname{vol}_{g(s)}\left(B_{g(s)}\left(x_{k},\,\frac{\varepsilon}{2}\right)\cap\phi_{k}(C_{k})\right)\geq \frac{\pi}{4}\left(\frac{\varepsilon}{2}\right)^{2}=\frac{\pi\varepsilon^{2}}{16}.
\end{equation*}
As a consequence, we obtain the following uniform lower bound on $\operatorname{vol}_{g_{T}}(\phi_{k}(C_{k}))$:
\begin{equation}\label{volume}
\operatorname{vol}_{g_{T}}(\phi_{k}(C_{k}))=\lim_{s\to T^{-}}\operatorname{vol}_{g(s)}(\phi_{k}(C_{k}))
\geq\frac{\pi\varepsilon^{2}}{16}.
\end{equation}
To conclude, substitute expression \eqref{contradiction} into \eqref{limitt}, then use the lower bound
\eqref{volume}, and finally let $k\to\infty$. This gives the desired contradiction.

Conversely, suppose that $(N,\,h)$ is the shrinking gradient K\"ahler-Ricci soliton of \cite{FIK} on the blowup of
$\mathbb{C}^{2}$ at the origin and for sake of a contradiction, assume that $\lim_{t\to T^{-}}\vol_{g(t)}(M)=0$. Then
\cite{tosatti10} tells us that $M$ exhibits the structure of a
Fano fibration $\pi:M\to B$ over a base $B$, where in particular $-K_{M}$ is $\pi$-ample. If $B$ is a point, then
$M$ is a del Pezzo surface and \cite{tosatti10} (see also \cite{song23}) further tells us
that the K\"ahler class of the initial metric $g(0)$ is $c_{1}(M)$ and that the diameter
$\operatorname{diam}(M,\,g(t))$ of $(M, g(t))$ tends to zero as $t\to T$. In fact, the work of
Perelman (see \cite{sesum1}) gives us the upper bound $\operatorname{diam}(M,\,g(t))\leq C(T-t)^{\frac{1}{2}}$, which,
for the re-scaled limit $g_{j}(t)$, $t<0$, translates to $\operatorname{diam}(M,\,g_{j}(t))\leq Ct$. This latter bound implies that
$(N,\,h)$ is compact which yields a contradiction. Hence we conclude that $B$ is one-dimensional.
The fact that $-K_{M}$ is $\pi$-ample now tells us that the generic fibre of $\pi:M\to B$ is a
holomorphic $\mathbb{P}^{1}$. Furthermore, by considering the minimal model of $M$ and using the $\pi$-ampleness of $-K_{M}$,
Claim \ref{continuity} applies with $\varpi(K)$ replaced by $M$ and $[C]$ replaced by the homology class of a $\mathbb{P}^{1}$-fibre
of the fibration $\pi:M\to B$. The result is that the singular fibres of $M$ comprise a bubble tree of two $(-1)$-curves.

Now, recalling the setup outlined at the beginning of this section, let $B_{R}:=B_{h}(p,\,R)$ denote the ball of radius $R>0$
centred at $p$ with respect to $h$. Then for
all $R>0$ and $\delta\in(0,\,1)$, there exist diffeomorphisms
$\phi_{k}:\overline{B_{R}}\to M$ with $\phi_{k}(p)=x$ such that $\phi_{k}^{*}g_{k}(t)\to h(t)$ with derivatives on $\overline{B_{R}}$ as $k\to\infty$ for all $t\in[-1,\,-\delta]$.
Let $\widetilde{E}$ denote the exceptional curve in $N$ and choose $R$ sufficiently large, $R=R_{1}$ say, so that
$\widetilde{E}\subset B_{R_{1}}$. Then since $\phi_{k}^{*}J$
converges smoothly locally to $\widetilde{J}$ as $k\to+\infty$, we can,
by Corollary \ref{curve2}, construct a $\phi_{k}^{*}J$-holomorphic curve $E_{k}$ in $B_{R_{1}}$ for each $k$ sufficiently large such that
$E_{k}\to\widetilde{E}$ in $C^{0}$ as $k\to+\infty$.

Recall from \cite{FIK} that the soliton $h=h(-1)$ lives on $\mathbb{C}^{2}$ blown up at a point,
is $U(2)$-invariant, and is asymptotic to a K\"ahler cone metric on $\mathbb{C}^{2}$. As such, for all $\lambda>0$, there exists a compact subset $K_{\lambda}\subset N$ containing $\widetilde{E}$ in the interior such that for all $x\in N\setminus K_{\lambda}$, $\operatorname{inj}_{h}(x)\geq3\lambda$ and $\sup_{B_{h}(x,\,\operatorname{inj}_{h}(x))}|\operatorname{Rm}(h)|_{h}\leq\frac{\pi^{2}}{3\lambda^{2}}$. Set $\lambda=4$, take the corresponding $K_{\lambda}$, and
choose $x\in N\setminus K_{\lambda}$ with $|x|=\hat{R}$ for $\hat{R}>0$ to be chosen later. By the $U(2)$-invariance of $h$, the aforementioned bounds on the injectivity radius and curvature hold
at all points on the sphere $\{|z|=\hat{R}\}$. Choose $\hat{R}$ sufficiently large so that $B_{R_{1}}\subset\{|z|\leq\hat{R}\}$
and such that $\overline{B_{h}(y,\,3\lambda)}\cap\widetilde{E}=\emptyset$ for all $y\in \{|z|=\hat{R}\}$. Next, choose $R>R_{1}$
sufficiently large so that $\{|z|\leq\hat{R}\}\subset B_{R}$ and so that $B_{R}$ contains $\overline{B_{h}(y,\,3\lambda)}$ for every
$y\in\{|z|=\hat{R}\}$. Finally, fix $k$ (depending on $R$) sufficiently large so that $\phi_{k}^{*}g_{k}(-1)$ is sufficiently close to $h$ in derivatives
to guarantee that for all $y\in\{|z|=\hat{R}\}$,
\begin{enumerate}
  \item $\operatorname{inj}_{\phi_{k}^{*}g_{k}(-1)}(y)\geq2\lambda$,
  \item $B_{\phi_{k}^{*}g_{k}(-1)}(y,\,2\lambda)\subset B_{h}(y,\,3\lambda)$,
  \item $\sup_{B_{\phi_{k}^{*}g_{k}(-1)}(y,\,2\lambda)}|\operatorname{Rm}(\phi_{k}^{*}g_{k}(-1))|_{\phi_{k}^{*}g_{k}(-1)}\leq\frac{\pi^{2}}{2\lambda^{2}}$.
\end{enumerate}
As a consequence of (ii), by choosing $k$ larger if necessary, we may assume in addition that for all $y\in\{|z|=\hat{R}\}$,
\begin{enumerate}\setcounter{enumi}{3}
  \item $\overline{B_{\phi_{k}^{*}g_{k}(-1)}(y,\,2\lambda)}\cap E_{k}=\emptyset$ and $\overline{B_{\phi_{k}^{*}g_{k}(-1)}(y,\,2\lambda)}\cap\partial B_{R}=\emptyset$.
\end{enumerate}

Now, $\phi_{k}(E_{k})$ will comprise one of the components of the bubble tree of the two $(-1)$-curves
in some exceptional fibre of the fibration $\pi:M\to B$. Write $E_{(1)}:=\phi_{k}(E_{k})$ and
let $E_{(2)}$ denote the other component. Then $\phi_{k}^{-1}(E_{(2)}\cap\phi_{k}(\overline{B_{R}}))$
defines a real surface in $\overline{B_{R}}$ intersecting $E_{k}$ at precisely one point. Let
$S_{k}\subset\overline{B_{R}}$ denote the unique connected component of this real surface intersecting $E_{k}$.
Then $S_{k}\cap\partial B_{R}\neq\emptyset$,
for otherwise $S_{k}$ would be contained in $B_{R}$
defining a $\phi_{k}^{*}J$-holomorphic $\mathbb{P}^{1}$ which, using Corollary \ref{curve2}, could be perturbed to a $\widetilde{J}$-holomorphic curve in $N$ distinct from $\widetilde{E}$ (after choosing $k$ larger if necessary), thereby leading to a contradiction. In particular, it follows that $S_{k}$ must intersect the hypersurface $\{|z|=\hat{R}\}$ at some point $q$.
Take the unique connected component $S_{k}^{q}\subset B_{\phi_{k}^{*}g_{k}(-1)}(q,\,2\lambda)$ of $S_{k}\cap B_{\phi_{k}^{*}g_{k}(-1)}(q,\,2\lambda)$ passing through $q$.
Clearly, if non-empty, the connected components of the boundary $\partial S_{k}$ are contained in $\partial B_{R}$. Thus, from (iv) above it follows that
$\partial S_{k}^{q}\subset\partial B_{\phi_{k}^{*}g_{k}(-1)}(q,\,2\lambda)$.
Next recalling points (i) and (iii) above, after unravelling the definitions and noting that
$\phi_{k}(S_{k}^{q})$ is $J$-holomorphic,
an application of \cite[Comment 1, p.178, and Proposition 4.3.1(ii)]{sikorav}
(see also \cite[Lemma 5.2]{collins100}) allows us to assert that
$$\vol_{g_{k}(-1)}(\phi_{k}(S^{q}_{k})\cap B_{g_{k}(-1)}(\phi_{k}(q),\,r))\geq\frac{\pi r^{2}}{4}$$
for all $0<r<2\lambda$. Set $r=\lambda=4$. Then we find that
$$\vol_{g_{k}(-1)}(\phi_{k}(S^{q}_{k})\cap B_{g_{k}(-1)}(\phi_{k}(q),\,4))\geq4\pi,$$
which, as $\phi_{k}(S^{q}_{k})\cap B_{g_{k}(-1)}(\phi_{k}(q),\,4)\subseteq E_{(2)}$, leads to
the lower bound $$\vol_{g_{k}(-1)}(E_{(2)})\geq4\pi.$$
On the other hand, using \cite[equation (1.2)]{tosatti10} and computing as in \eqref{contradiction} with $t=-1$, keeping in mind the fact that $(E_{(2)})^{2}=-1$, we derive that
$$\operatorname{vol}_{g_{k}(-1)}(E_{(2)})=2\pi\int_{E_{(2)}}c_{1}(-K_{M}|_{E_{2}})=2\pi.$$
This is a contradiction. We therefore conclude that $\lim_{t\to T^{-}}\vol_{g(t)}(M)>0$, as desired.

\bibliographystyle{amsalpha}

\bibliography{ref2}

\newcommand{\etalchar}[1]{$^{#1}$}
\def\cprime{$'$} \def\cprime{$'$}
\providecommand{\bysame}{\leavevmode\hbox to3em{\hrulefill}\thinspace}
\providecommand{\MR}{\relax\ifhmode\unskip\space\fi MR }
\providecommand{\MRhref}[2]{%
  \href{http://www.ams.org/mathscinet-getitem?mr=#1}{#2}
}
\providecommand{\href}[2]{#2}
\begin{thebibliography}{CCG{\etalchar{+}}07}

\bibitem[BEG13]{Bou-Eys-Gue}
S.~Boucksom, P.~Eyssidieux, and V.~Guedj, \emph{An introduction to the
  {K}\"ahler-{R}icci flow}, Lecture Notes in Mathematics, vol. 2086, Springer,
  Cham, 2013, pp.~viii+333. \MR{3202578}

\bibitem[Bry08]{Bry-Kah-Sol}
R.~Bryant, \emph{Gradient {K}\"ahler {R}icci solitons}, Ast\'erisque (2008),
  no.~321, 51--97, G{\'e}om{\'e}trie diff{\'e}rentielle, physique
  math{\'e}matique, math{\'e}matiques et soci{\'e}t{\'e}. I. \MR{2521644}

\bibitem[BT82]{Bott}
R.~Bott and L.~Tu, \emph{Differential forms in algebraic topology}, Graduate
  Texts in Mathematics, vol.~82, Springer-Verlag, New York, 1982. \MR{658304
  (83i:57016)}

\bibitem[CCD]{ccd2}
C.~Cifarelli, R.~J. Conlon, and A.~Deruelle, \emph{An {A}ubin continuity path
  for shrinking gradient {K}\"ahler-{R}icci solitons}, \textnormal{in
  preparation}.

\bibitem[CCG{\etalchar{+}}07]{chow}
B.~Chow, S.-C. Chu, D.~Glickenstein, C.~Guenther, J.~Isenberg, T.~Ivey,
  D.~Knopf, P.~Lu, F.~Luo, and L.~Ni, \emph{The {R}icci flow: techniques and
  applications. {P}art {I}}, Mathematical Surveys and Monographs, vol. 135,
  American Mathematical Society, Providence, RI, 2007, Geometric aspects.
  \MR{2302600}

\bibitem[CDS19]{cds}
R.~J. Conlon, A.~Deruelle, and S.~Sun, \emph{Classification results for
  expanding and shrinking gradient {K}\"ahler-{R}icci solitons},
  \textnormal{arXiv:1904.00147} (2019).

\bibitem[Cif20]{charlie}
C.~Cifarelli, \emph{Uniqueness of shrinking gradient {K}\"ahler-{R}icci
  solitons on non-compact toric manifolds}, \textnormal{arXiv:2010.00166}
  (2020).

\bibitem[CJL21]{collins100}
T.~Collins, A.~Jacob, and Y.-S. Lin, \emph{Special {L}agrangian submanifolds of
  log {C}alabi-{Y}au manifolds}, Duke Math. J. \textbf{170} (2021), no.~7,
  1291--1375. \MR{4255060}

\bibitem[CK04]{Chowchow}
B.~Chow and D.~Knopf, \emph{The {R}icci flow: an introduction}, Mathematical
  Surveys and Monographs, vol. 110, American Mathematical Society, Providence,
  RI, 2004. \MR{2061425}

\bibitem[CLS11]{cox}
D.~Cox, J.~Little, and H.~Schenck, \emph{Toric varieties}, Graduate Studies in
  Mathematics, vol. 124, American Mathematical Society, Providence, RI, 2011.
  \MR{2810322}

\bibitem[CST09]{chen-soliton}
X.~Chen, S.~Sun, and G.~Tian, \emph{A note on {K}\"{a}hler-{R}icci soliton},
  Int. Math. Res. Not. IMRN (2009), no.~17, 3328--3336. \MR{2535001}

\bibitem[CZ10]{caoo}
H.-D. Cao and D.~Zhou, \emph{On complete gradient shrinking {R}icci solitons},
  J.~Differ.~Geom. \textbf{85} (2010), no.~2, 175--185. \MR{2732975}

\bibitem[EMT11]{topping}
J.~Enders, R.~M\"uller, and P.~Topping, \emph{On type-{I} singularities in
  {R}icci flow}, Comm. Anal. Geom. \textbf{19} (2011), no.~5, 905--922.
  \MR{2886712}

\bibitem[FIK03]{FIK}
M.~Feldman, T.~Ilmanen, and D.~Knopf, \emph{Rotationally symmetric shrinking
  and expanding gradient {K}\"ahler-{R}icci solitons}, J.~Differ.~Geom.
  \textbf{65} (2003), no.~2, 169--209. \MR{2058261}

\bibitem[FMZ08]{fang}
F.-Q. Fang, J.-W. Man, and Z.-L. Zhang, \emph{Complete gradient shrinking
  {R}icci solitons have finite topological type}, C. R. Math. Acad. Sci. Paris
  \textbf{346} (2008), no.~11-12, 653--656. \MR{2423272}

\bibitem[Fra59]{frankel}
T.~Frankel, \emph{Fixed points and torsion on {K}\"{a}hler manifolds}, Ann. of
  Math. (2) \textbf{70} (1959), 1--8. \MR{0131883}

\bibitem[Fut88]{fut2}
A.~Futaki, \emph{K\"ahler-{E}instein metrics and integral invariants}, Lecture
  Notes in Mathematics, vol. 1314, Springer-Verlag, Berlin, 1988. \MR{947341}

\bibitem[Gro85]{gromov}
M.~Gromov, \emph{Pseudo holomorphic curves in symplectic manifolds}, Invent.
  Math. \textbf{82} (1985), no.~2, 307--347. \MR{809718}

\bibitem[GS16]{guo}
B.~Guo and J.~Song, \emph{On {F}eldman-{I}lmanen-{K}nopf's conjecture for the
  blow-up behavior of the {K}\"ahler {R}icci flow}, Math. Res. Lett.
  \textbf{23} (2016), no.~6, 1681--1719. \MR{3621103}

\bibitem[Gui94]{gilly}
V.~Guillemin, \emph{Kaehler structures on toric varieties}, J. Differential
  Geom. \textbf{40} (1994), no.~2, 285--309. \MR{1293656}

\bibitem[Ham82]{hamilton}
R.~Hamilton, \emph{Three-manifolds with positive {R}icci curvature},
  J.~Differ.~Geom. \textbf{17} (1982), no.~2, 255--306. \MR{664497 (84a:53050)}

\bibitem[Iwa49]{iwasawa}
K.~Iwasawa, \emph{On some types of topological groups}, Ann. of Math. (2)
  \textbf{50} (1949), 507--558. \MR{0029911}

\bibitem[Kob58]{koby}
S.~Kobayashi, \emph{Fixed points of isometries}, Nagoya Math. J. \textbf{13}
  (1958), 63--68. \MR{103508}

\bibitem[Kod63]{kodaira1}
K.~Kodaira, \emph{On stability of compact submanifolds of complex manifolds},
  American Journal of Mathematics \textbf{85} (1963), no.~1, 79--94.

\bibitem[LTZ18]{li1}
Y.~Li, G.~Tian, and X.~Zhu, \emph{Singular limits of {K}\"ahler-{R}icci flow on
  {F}ano {G}-manifolds}, \textnormal{arXiv:1807.09167} (2018).

\bibitem[M\'14]{maximo}
D.~M\'{a}ximo, \emph{On the blow-up of four-dimensional {R}icci flow
  singularities}, J. Reine Angew. Math. \textbf{692} (2014), 153--171.
  \MR{3274550}

\bibitem[MS94]{dusa2}
D.~McDuff and D.~Salamon, \emph{{$J$}-holomorphic curves and quantum
  cohomology}, University Lecture Series, vol.~6, American Mathematical
  Society, Providence, RI, 1994. \MR{1286255}

\bibitem[MS04]{dusa1}
\bysame, \emph{{$J$}-holomorphic curves and symplectic topology}, American
  Mathematical Society Colloquium Publications, vol.~52, American Mathematical
  Society, Providence, RI, 2004. \MR{2045629}

\bibitem[MW15]{munteanu}
O.~Munteanu and J.~Wang, \emph{Topology of {K}\"ahler {R}icci solitons},
  J.~Differ.~Geom. \textbf{100} (2015), no.~1, 109--128. \MR{3326575}

\bibitem[MW19]{wang22}
\bysame, \emph{Structure at infinity for shrinking {R}icci solitons}, Ann. Sci.
  \'{E}c. Norm. Sup\'{e}r. (4) \textbf{52} (2019), no.~4, 891--925.
  \MR{4038455}

\bibitem[Nab10]{naber}
A.~Naber, \emph{Noncompact shrinking four solitons with nonnegative curvature},
  J. Reine Angew. Math. \textbf{645} (2010), 125--153. \MR{2673425}

\bibitem[PRS11]{pigola1}
S.~Pigola, M.~Rimoldi, and A.~Setti, \emph{Remarks on non-compact gradient
  {R}icci solitons}, Math. Z. \textbf{268} (2011), no.~3-4, 777--790.
  \MR{2818729}

\bibitem[PW94]{wu}
E.~Prato and S.~Wu, \emph{Duistermaat-{H}eckman measures in a non-compact
  setting}, Compositio Math. \textbf{94} (1994), no.~2, 113--128. \MR{1302313}

\bibitem[Rua99]{ruannn}
W.-D. Ruan, \emph{On the convergence and collapsing of {K}\"{a}hler metrics},
  J. Differential Geom. \textbf{52} (1999), no.~1, 1--40. \MR{1743466}

\bibitem[Shi89]{shii}
W.-X. Shi, \emph{Deforming the metric on complete {R}iemannian manifolds}, J.
  Differential Geom. \textbf{30} (1989), no.~1, 223--301. \MR{1001277}

\bibitem[Sik94]{sikorav}
J.-C. Sikorav, \emph{Some properties of holomorphic curves in almost complex
  manifolds}, Holomorphic curves in symplectic geometry, Progr. Math., vol.
  117, Birkh\"{a}user, Basel, 1994, pp.~165--189. \MR{1274929}

\bibitem[Son14]{song23}
J.~Song, \emph{Finite-time extinction of the {K}\"{a}hler-{R}icci flow}, Math.
  Res. Lett. \textbf{21} (2014), no.~6, 1435--1449. \MR{3335855}

\bibitem[Son15]{sing}
\bysame, \emph{Some type {I} solutions of {R}icci flow with rotational
  symmetry}, Int. Math. Res. Not. IMRN (2015), no.~16, 7365--7381. \MR{3428966}

\bibitem[ST08]{sesum1}
N.~Sesum and G.~Tian, \emph{Bounding scalar curvature and diameter along the
  {K}\"{a}hler {R}icci flow (after {P}erelman)}, J. Inst. Math. Jussieu
  \textbf{7} (2008), no.~3, 575--587. \MR{2427424}

\bibitem[Ste44]{steeny}
N.~Steenrod, \emph{The classification of sphere bundles}, Ann. of Math. (2)
  \textbf{45} (1944), 294--311. \MR{9857}

\bibitem[Top06]{topping-book}
P.~Topping, \emph{Lectures on the {R}icci flow}, London Mathematical Society
  Lecture Note Series, vol. 325, Cambridge University Press, Cambridge, 2006.
  \MR{2265040}

\bibitem[TZ02]{Tian-Zhu-II}
G.~Tian and X.~Zhu, \emph{A new holomorphic invariant and uniqueness of
  {K}\"ahler-{R}icci solitons}, Comment. Math. Helv. \textbf{77} (2002), no.~2,
  297--325. \MR{1915043}

\bibitem[TZ18]{tosatti10}
V.~Tosatti and Y.~Zhang, \emph{Finite time collapsing of the
  {K}\"{a}hler-{R}icci flow on threefolds}, Ann. Sc. Norm. Super. Pisa Cl. Sci.
  (5) \textbf{18} (2018), no.~1, 105--118. \MR{3783785}

\bibitem[Wyl08]{wyliee}
W.~Wylie, \emph{Complete shrinking {R}icci solitons have finite fundamental
  group}, Proc. Amer. Math. Soc. \textbf{136} (2008), no.~5, 1803--1806.
  \MR{2373611}

\bibitem[WZ04]{soliton}
X.-J. Wang and X.~Zhu, \emph{K\"ahler-{R}icci solitons on toric manifolds with
  positive first {C}hern class}, Adv. Math. \textbf{188} (2004), no.~1,
  87--103. \MR{2084775 (2005d:53074)}

\bibitem[Zha09]{Zhang-Com-Ricci-Sol}
Z.-H. Zhang, \emph{On the completeness of gradient {R}icci solitons}, Proc.
  Amer. Math. Soc. \textbf{137} (2009), no.~8, 2755--2759. \MR{2497489}

\bibitem[Zhu00]{zhuu}
X.~Zhu, \emph{K\"{a}hler-{R}icci soliton typed equations on compact complex
  manifolds with {$C_1(M)>0$}}, J. Geom. Anal. \textbf{10} (2000), no.~4,
  759--774. \MR{1817785}

\end{thebibliography}

\end{document}